\documentclass{amsart}
\usepackage{amssymb,amsmath,amsthm,amscd}
\usepackage{mathtools}
\usepackage{esint}
\usepackage{mathrsfs}
\usepackage{mathabx}
\usepackage[T1]{fontenc}
\usepackage{hyperref}
\hypersetup{colorlinks,linkcolor={red},citecolor={blue}} 
\usepackage{float}
\usepackage{tikz,pgfplots}
\usetikzlibrary{cd}
\usetikzlibrary{patterns,arrows,decorations.pathreplacing,calc}
\usetikzlibrary{decorations.markings}

\pgfplotsset{compat=1.14}

\newtheorem{theorem}{Theorem}
\newtheorem{lemma}[theorem]{Lemma}
\newtheorem{corollary}[theorem]{Corollary}

\newtheorem{definition}[theorem]{Definition}

\newtheorem{remark}[theorem]{Remark}

\numberwithin{theorem}{section}
\numberwithin{equation}{section}

\def\N {{\mathbb N}}

\def\R {{\mathbb R}}

\DeclareMathOperator{\diff}{d\!}
\DeclareMathOperator{\ch}{ch}
\DeclareMathOperator{\op}{op}

\DeclarePairedDelimiter\abs{\lvert}{\rvert}
\DeclarePairedDelimiter\norm{\lVert}{\rVert}

\begin{document}
	\title{Uniform Fourier restriction for convex curves}
	\date{\today}
	
	\author{Marco Fraccaroli}
	\address{Mathematisches Institut, Universit\"{a}t Bonn, Endenicher Allee 60, 53115 Bonn, Germany}
	\email{mfraccar@math.uni-bonn.de}
	\address{BCAM - Basque Center for Applied Mathematics, Mazarredo, 14 E48009 Bilbao, Basque Country – Spain.}
	\email{mfraccaroli@bcamath.org}
	
	
	\begin{abstract}
		{We extend the estimates for maximal Fourier restriction operators proved by M\"{u}ller, Ricci, and Wright in \cite{MR3960255} and Ramos in \cite{MR4055940} to the case of arbitrary convex curves in the plane, with constants uniform in the curve. The improvement over M\"{u}ller, Ricci, and Wright and Ramos is given by the removal of the $\mathcal{C}^2$ regularity condition on the curve. This requires the choice of an appropriate measure for each curve, that is suggested by an affine invariant construction of Oberlin in \cite{MR1960918}. As corollaries, we obtain a uniform Fourier restriction theorem for arbitrary convex curves and a result on the Lebesgue points of the Fourier transform on the curve. }
	\end{abstract}
	
	\maketitle

\section{Introduction}

The study of the restriction phenomena for the Fourier transform in $\R^n$ has been an active research topic in harmonic analysis over the last decades. The most common instance of it, a \emph{Fourier restriction estimate}, comes in the form of the following inequality for every Schwartz function $f \in \mathcal{S}(\R^n)$
\begin{equation*} 
\norm{\widehat{f}_{|_S}}_{L^q(S,\nu)} \leq C(p,q,S,\nu) \norm{f}_{L^p(\R^n)},
\end{equation*}
where $\widehat{f}$ is the Fourier transform of $f$, $S$ a hypersurface with appropriate curvature properties, $\nu$ a suitable measure on $S$, the exponents $p$ and $q$ vary in an appropriate range, and the constant $C(p,q,S,\nu)$ is independent of $f$. The a priori estimate in the previous display guarantees the existence of a bounded restriction operator $\mathcal{R} \colon L^p(\R^n) \to L^q(S,\nu)$ such that $\mathcal{R} f = \widehat{f}$ on $S$ when $f \in \mathcal{S}(\R^n)$. Such Fourier restriction estimates were first studied by Fefferman and Stein who proved a result in any dimension (\cite{MR257819}, pg. 28). This result was later improved by the celebrated Stein-Tomas method (\cite{MR1232192}, \cite{MR358216}) which focuses on the case $q=2$. Since then, a huge mathematical effort has been put into studying the Fourier restriction phenomena leading to the development of many new techniques. Despite that, many problems for any arbitrary dimension $n \geq 3$ are still open. For example, the question about sufficient conditions on the exponents $p$ and $q$ in order for a Fourier restriction estimate to hold true. 

In fact, standard examples (constant functions, Knapp examples) in the case of the sphere $S=\mathbb{S}^{n-1}$ with the induced Borel measure $\sigma$  provide necessary conditions on the range of exponents $p$ and $q$ in order for the inequality in the previous display to hold true, namely
\begin{equation*}
1 \leq p < \frac{2n}{n+1}, \qquad \qquad q \leq \frac{n-1}{n+1} p',
\end{equation*}
where $\frac{1}{p}+\frac{1}{p'}=1$. The main conjecture in the theory of Fourier restriction is that these conditions are sufficient too. We refer to the exposition of Tao in \cite{MR2087245} for a description of the aforementioned standard examples. We point to the same reference also for a more exhaustive introduction to the research topic of Fourier restriction, as well as an overview of the results up to 2004. 

In the case of a $\mathcal{C}^2$ convex curve $\Gamma$ in the plane $\R^2$ the conditions on the exponents are also sufficient. Sharp estimates were proved first for the circle $\mathbb{S}^1$ by Zygmund in \cite{MR387950}, and for more general curves by Carleson and Sj\"{o}lin in \cite{MR361607} and Sj\"{o}lin in \cite{MR385437}. In fact, in \cite{MR385437} Sj\"{o}lin proved a uniform Fourier restriction result for such curves upon the choice of a specific measure $\nu = \nu(\Gamma)$ on each curve. This is the so called affine arclength measure, encompassing the curvature properties of the $\Gamma$. In the case of the circle, it coincides with the induced Borel measure $\sigma$, thus proving the sharpness of the result of Sj\"{o}lin.

In \cite{MR3960255} M\"{u}ller, Ricci, and Wright addressed a different feature of the Fourier restriction phenomena, namely the pointwise relation between $\mathcal{R}f$ and $\widehat{f}$ for an arbitrary function $f \in L^p(\R^n)$. In the case of a $\mathcal{C}^2$ convex curve and a function $f \in L^p(\R^2)$, with $1 \leq p < 8/7$ they proved that $\nu$-almost every point of the curve is a Lebesgue point for $\widehat{f}$. Moreover, they showed that the regularized value of $\widehat{f}$ coincides with that of $\mathcal{R}f$ at $\nu$-almost every point of the curve. The main ingredient in their proof is given by the estimates for a certain maximal Fourier restriction operator $\mathcal{M}$ defined as follows. For every Schwartz function $f \in \mathcal{S}(\R^2)$ we define
\begin{equation} \label{eq:maxopr_1}
\mathcal{M} \widehat{f} (x) \coloneqq  \sup_{R} \abs[\Big]{ \int_{\R^2} \widehat{f} (x-y) \chi_R(y) \diff y},
\end{equation}
where $\chi_R$ is a bump function adapted to $R$ normalized in $L^1(\R^2)$ and the supremum is taken over all rectangles $R$ centred at the origin with sides parallel to the axes. Next, they use the estimate 
\begin{equation*}
M \widehat{f} \leq (\mathcal{M} \widehat{h})^{\frac{1}{2}},
\end{equation*}
where $M$ is the classical two-parameter maximal operator and $h$ is defined by $\widehat{h} = \abs{\widehat{f}}^2$. To obtain the desired result about the Lebesgue points for $f \in L^p(\R^2)$, they need to bound the norms of $h$ by those of $f$. This forces to assume the additional condition $p < 8/7$ on the exponent.

In \cite{MR4055940} Ramos extended their result to the full range $1 \leq p < 4/3$ in the case of the circle $\mathbb{S}^1$. The improvement relies on the estimates he proved for a more general class of maximal Fourier restriction operators 
\begin{equation*}
\{ \mathcal{M}_g \colon \norm{g}_{L^\infty(\R^2)} = 1 \},
\end{equation*}
where for every function $g$ normalized in $L^{\infty}(\R^2)$ we define $\mathcal{M}_g$ as follows. For every Schwartz function $f \in \mathcal{S}(\R^2)$ we define 
\begin{equation} \label{eq:maxopr_2}
	\mathcal{M}_g \widehat{f} (x) \coloneqq  \sup_{R} \abs[\Big]{ \int_{\R^2} \widehat{f} (x-y) g(x-y) \abs{R}^{-1} 1_R(y) \diff y},
\end{equation}
where the supremum is taken over all rectangles $R$ centred at the origin with sides parallel to the axes. In particular, the freedom in the choice of $g$ allows Ramos to bring the absolute value inside the integral defining the averages, thus bypassing the artificial limitation arising in M\"{u}ller, Ricci, and Wright argument.

The line of investigation about the boundedness properties of maximal Fourier restriction operators initiated by M\"{u}ller, Ricci, and Wright has been developed further in a series of papers that followed up. In \cite{MR4541301} Vitturi studied estimates for a maximal Fourier restriction operator in the case of the sphere $\mathbb{S}^{n-1}$ in $\R^n$ for any arbitrary dimension $n \geq 3$. The operator considered is of the form described in \eqref{eq:maxopr_1} with the supremum taken over averages on balls. Vitturi used the estimates on this operator to prove the analogue of the Lebesgue points property of $\widehat{f}$ for every function $f \in L^p(\R^2)$ with $1 \leq p \leq 8/7$. The range of exponents was later improved by Ramos in \cite{MR4055940} to $1 \leq p \leq 4/3$ considering maximal Fourier restriction operators of the form described in \eqref{eq:maxopr_2} with the supremum taken over averages on balls. It is worth noting that in the case of dimension $n \geq 3$, due to the range of Stein-Tomas estimates, the endpoint $p = 4/3$ is recovered, as opposed to the case of dimension $n=2$. 

In parallel, in \cite{MR4001074} Kova\v{c} studied estimates for certain variational Fourier restriction operators in any arbitrary dimension $n \geq 2$. These operators are defined by variation norms, rather than the $L^\infty$ norm, on averages of the form of those appearing on the right hand side in \eqref{eq:maxopr_1} computed with respect to isotropic rescaling of an arbitrary measure $\mu$. He developed an abstract method to upgrade Fourier restriction estimates with $p<q$ to estimates for the variational Fourier restriction operators with the same exponents. As a consequence, he obtained a quantitative version of the qualitative result about the convergence of averages in Lebesgue points. Kova\v{c} provided sufficient conditions for the method to work. These conditions are expressed in terms of certain decay estimates on the gradient of $\widehat{\mu}$. Together with Oliveira e Silva, he later improved over the sufficient conditions in \cite{MR4282056}. 

Next, in \cite{MR4395599} Ramos studied estimates for certain maximal Fourier restriction operators associated with an arbitrary measure $\mu$ in the case of dimension $n=2$ and $n=3$. Once again, he considered operators of the form described in \eqref{eq:maxopr_2} with the supremum taken over averages computed with respect to isotropic rescaling of $\mu$. Ramos provided sufficient conditions on the measure $\mu$ to obtain estimates for the maximal Fourier restriction operators. These conditions are expressed in terms of the boundedness properties close to $L^2(\R^n)$ of the maximal function associated with $\mu$. In particular, he recovered the case of the spherical measures that, in dimension $n=2$ and $n=3$, do not satisfy the sufficient conditions stated in \cite{MR4001074, MR4282056}. Since Kova\v{c} and Oliveira e Silva use stronger norms but weaker averages than Ramos, the results in \cite{MR4001074, MR4282056} and those in \cite{MR4055940, MR4395599} are not comparable, and we refer to those papers for an exposition of the connections between their results.

Finally, in \cite{MR4446236} Jesurum studied estimates for a maximal Fourier restriction operator in the case of the moment curve $\{ (t,\frac{1}{2}t^2, \dots, \frac{1}{n} t^n) \colon t \in \R \}$ in $\R^n$ for any arbitrary dimension $n \geq 3$. The operator considered is of the form described in \eqref{eq:maxopr_2} with the supremum taken over averages on balls. Jesurum followed the argument of Drury in \cite{MR781781}, where Drury proved Fourier restriction estimates for the moment curve in the full range $1 \leq p < (n^2+n+2)/(n^2+n)$, $q = 2p'/(n^2+n)$. In particular, Jesurum recovered the analogue of the Lebesgue points property of $\widehat{f}$ for every function $f \in L^p(\R^n)$ with $p$ in the same range of exponents.

In fact, both Ramos in \cite{MR4395599} and Jesurum in \cite{MR4446236} considered also stronger maximal Fourier restriction operators. In particular, in the definition of these operators they substituted the supremum taken over $L^1$ averages on balls with that over $L^r$ averages for arbitrary $r \geq 1$. By H\"{o}lder's inequality, the operators are increasing in $r$. We refer to those papers for details about the estimates for these maximal Fourier restriction operators, as well as the analysis of the threshold values for $r \geq 1$ in relation to such estimates.

In this paper, we are concerned with extending the results of M\"{u}ller, Ricci, and Wright in \cite{MR3960255} and Ramos in \cite{MR4055940} to the case of arbitrary \emph{convex curves} in the plane, uniformly in the curve. Such curves are the boundaries of non-empty open convex sets in $\R^2$. Passing from the case of the circle $\mathbb{S}^1$ to the case of an arbitrary $\mathcal{C}^2$ convex curve $\Gamma$ is straight-forward upon the choice of the affine arclength measure on $\Gamma$. We are going to introduce such measure in a moment. The main point of the paper is the removal of the $\mathcal{C}^2$ regularity condition on the curve. It goes through the choice of a suitable extension of the affine arclength measure, which was suggested by an affine invariant construction described by Oberlin in \cite{MR1960918}. The desired extension of the results then follows the line of proof by Ramos up to the appropriate modifications. 

We turn now to the description of two measures on an arbitrary convex curve $\Gamma$ in the plane. We elaborate in more detail in Section~\ref{sec:preliminaries} and Appendix~\ref{sec:appendix}. A first measure $\nu$ is built from the arclength parametrization such a curve always admits
\begin{equation*}
z \colon J \to \Gamma \subseteq \R^2,
\end{equation*}
where $J$ is an interval in $\R$, possibly degenerate. Let $m$ be the Lebesgue measure on $J$. The first and second derivatives $z'$ and $z''$ with respect to $m$ are functions well-defined pointwise $m$-almost everywhere on $J$. We define a measure $\nu$ on $J$ by
\begin{equation*}
\diff \nu(t) = \sqrt[3]{\det \begin{pmatrix} z(t) & z''(t) \end{pmatrix}} \diff t.
\end{equation*}
With a slight abuse of notation we denote by $\nu$ also its push-forward to $\Gamma$ via the affine arclength parametrization $z$. In particular, when $\Gamma$ is $\mathcal{C}^2$ the argument of the cubic root is well-defined everywhere in $J$ and the measure $\nu$ on $\Gamma$ is called \emph{affine arclength measure}. We extend the term to denote $\nu$ in the general case of arbitrary convex curves.

We define a second measure $\mu$ on $\Gamma$ following Oberlin. Oberlin's construction of the \emph{affine measures} $\{ \mu_{n,\alpha} \colon \alpha \geq 0 \}$ on $\R^n$ is analogous to that of the Hausdorff measures. The only difference is that in the former we use rectangular parallelepipeds in $\R^n$ to cover sets while in the latter we use balls. This change guarantees the affine invariance of $\mu_{n,\alpha}$, as well as it allows $\mu_{n,\alpha}$ to be sensitive to the curvature properties of the set on which $\mu_{n,\alpha}$ is evaluated. A general definition of $\mu_{n,\alpha}$ can be found in \cite{MR1960918}. Here, we restrict ourselves to the case $n=2$, $\alpha=2/3$ and we drop the subscripts from the notation of $\mu$.

\begin{definition} [Affine measure $\mu$ on $\R^2$]
	For every $\delta > 0$ and every subset $E \subseteq \R^2$ we define
	\begin{equation*}
	\mu^\delta(E) \coloneqq \inf \Big\{ \sum_{R \in \mathcal{R}'} \abs{R}^{\frac{1}{3}} \colon \mathcal{R}' \subseteq \mathcal{R}^\delta, E \subseteq \bigcup_{R \in \mathcal{R}'} R \Big\},
	\end{equation*}
	where $\abs{R}$ is the Lebesgue measure of the rectangle $R$ and $\mathcal{R}^\delta$ is the collection of all rectangles in $\R^2$ with diameter smaller than or equal to $\delta$.	Next, we define
	\begin{equation*}
	\mu^\ast(E) \coloneqq \lim_{\delta \to 0} \mu^\delta (E).
	\end{equation*}
	
	Finally, we define the \emph{affine measure $\mu$} on $\R^2$ to be the restriction of the outer measure $\mu^\ast$ on $\R^2$ to its Carath\'{e}odory measurable subsets of $\R^2$.
\end{definition}
With a slight abuse of notation we denote by $\mu$ also its restriction to the convex curve $\Gamma$, as well as its push-forward to $J$ via the inverse of the bijective function given by an arclength parametrization $z$ for $\Gamma$. 

In \cite{MR1960918} Oberlin proved that if the curve $\Gamma$ is $\mathcal{C}^2$, then the affine measure $\mu$ and the affine arclength measure $\nu$ are comparable up to multiplicative constants uniform in the curve. The first observation of this paper is the extension of this property to the case of arbitrary convex curves.
\begin{theorem} \label{thm:comparability}
	There exist constants $0 < A \leq B < \infty$ such that for every convex curve $\Gamma$ we have
	\begin{equation*}
	A \nu \leq \mu \leq B \nu,
	\end{equation*}
	where $\mu, \nu$ are the measures on $\Gamma$ defined above.
\end{theorem}

The second observation of this paper is the uniform extension of the boundedness properties of the maximal Fourier restriction operator defined in \eqref{eq:maxopr_2} to the case of arbitrary convex curves.
\begin{theorem} \label{thm:uniform2}
Let $1 \leq p < 4/3$, $q = p'/3$. There exists a constant $C = C(p) < \infty$ such that for every function $g \in L^{\infty}(\R^2)$ normalized in $L^{\infty}(\R^2)$, every convex curve $\Gamma$, and every Schwartz function $f \in \mathcal{S}(\R^2)$ we have
\begin{equation*}
	\norm{\mathcal{M}_g \widehat{f}}_{L^q(\Gamma,\nu)} \leq C \norm{f}_{L^p(\R^2)},
\end{equation*}
where $\nu$ is the measure on $\Gamma$ defined above.
\end{theorem}

We have two straight-forward corollaries. The first is a uniform Fourier restriction result for arbitrary convex curves. 
\begin{corollary} \label{thm:uniform}
	Let $1 \leq p < 4/3$, $q = p'/3$. There exists a constant $C = C(p) < \infty$ such that for every convex curve $\Gamma$ and every Schwartz function $f \in \mathcal{S}(\R^2)$ we have
	\begin{equation*}
	\norm{\widehat{f}}_{L^q(\Gamma,\nu)} \leq C	\norm{f}_{L^p(\R^2)},
	\end{equation*}
	where $\nu$ is the measure on $\Gamma$ defined above.
\end{corollary}
The second is the extension of the result on Lebesgue points of $\widehat{f}$ on the curve to the case of arbitrary convex curves.
\begin{corollary} \label{thm:lebsegue}
	Let $1 \leq p < 4/3$. Let $\Gamma$ be a convex curve and $\nu$ the measure on $\Gamma$ defined above. If $f \in L^p(\R^2)$, then $\nu$-almost every point of $\Gamma$ is a Lebesgue point for $\widehat{f}$. Moreover, the regularized value of $\widehat{f}$ coincides with the one of $\mathcal{R}f$ at $\nu$-almost every point of $\Gamma$.
\end{corollary}

The results stated in Theorem~\ref{thm:uniform2} and the corollaries highlight a strict relation between the following objects. On one hand, the affine arclength and Oberlin's affine measures, sensitive to the curvature properties of the sets on which they are defined. On the other hand, uniform estimates for classical operators involving smooth enough submanifolds in $\R^n$, where the curvature properties of the submanifold play a significant role. Beyond Fourier restriction operators, it is the case of convolution operators, X-ray transforms, and Radon transforms. We conclude the Introduction briefly mentioning previous works pointing at the aforementioned relation in the analysis of all these operators \cite{MR3103222, MR2460861, MR2503984, MR3181503, MR2916059, MR3292348, MR1049762, MR764500, jesurum2022fourier, MR1690999, MR1845006, MR3265961, MR3483472}. We refer to these papers and the references therein for a more thorough exposition of the relation. Finally, we point out the work of Gressman in \cite{MR3992033} on the generalization of the affine arclength measure to smooth enough submanifolds of any arbitrary dimension $d$ in $\R^n$.

\subsection*{Guide to the paper}

In Section~\ref{sec:preliminaries} we introduce some notations, definitions, and previous results we clarify in Appendix~\ref{sec:appendix}.
In Section~\ref{sec:affine_measure} we prove Theorem~\ref{thm:comparability}.
In Section~\ref{sec:proofs} we prove Theorem~\ref{thm:uniform2} and the corollaries.

\section*{Acknowledgements}
The author gratefully acknowledges financial support by the CRC 1060 \emph{The Mathematics of Emergent Effects} at the University of Bonn, funded through the Deutsche Forschungsgemeinschaft. He is also supported by the Basque Government through the BERC 2022-2025 program and by the Ministry of Science and Innovation: BCAM Severo Ochoa accreditation CEX2021-001142-S / MICIN / AEI / 10.13039/501100011033.

The author is thankful to Jo\~{a}o Pedro G. Ramos, Christoph Thiele and Gennady Uraltsev for helpful discussion, comments and suggestions that improved the exposition of the material, and for their support.

\section{Preliminaries} \label{sec:preliminaries}

\subsection{Notation} We introduce the following notations.
	
	For every interval $I \subseteq \R$ we denote by $\Delta(I)$ the lower triangle associated with $I$ defined by
	\begin{equation*}
	\Delta(I) \coloneqq \{(s,t)\in I\times I: t<s\} .
	\end{equation*}
	
	For all vectors $a,b \in \R^2 \setminus \{ (0,0) \}$ we denote by $\theta (a,b) \in [0,2 \pi)$ the counterclockwise angle from $a$ to $b$.

\subsection{Real analysis} 

We recall a result about the metric density associated with the absolutely continuous part of a measure with respect to the Lebesgue measure.

\begin{definition} \label{def:nicely}
	Let $x \in \R^n$. We say that a sequence $\{ E_\varepsilon \colon \varepsilon > 0 \}$ \emph{shrinks to $x$ nicely} if it is a sequence of Borel sets in $\R^n$ and there is a number $\alpha >0$ satisfying the following property. There is a sequence of balls $\{ B(x,r_\varepsilon) \colon \varepsilon > 0 \}$ with $\lim_{\varepsilon \to 0} r_\varepsilon =0$, such that for every $\varepsilon > 0$ we have $E_\varepsilon \subseteq B(x,r_\varepsilon)$ and
	\begin{equation*}
	m(E_\varepsilon(x)) \geq \alpha m(B(x,r_\varepsilon)).
	\end{equation*}
\end{definition}

\begin{theorem} [Rudin \cite{MR924157}, Theorem~7.14] \label{thm:derivative}
	For every $x \in \R^n$ let $\{ E_\varepsilon(x) \colon \varepsilon > 0 \}$ be a sequence that shrinks to $x$ nicely. Let $\mu$ be a Borel measure on $\R^n$. Let
	\begin{equation*}
	\diff \mu = \mu' \diff m + \diff \mu_s,
	\end{equation*}
	be the decomposition of $\mu$ into the absolutely continuous and singular parts with respect to the Lebesgue measure $m$ in $\R^n$. Then, for $m$-almost every $x \in \R^n$ we have
	\begin{equation*} 
	\lim_{\varepsilon \to 0} \frac{\mu(E_\varepsilon(x))}{m(E_\varepsilon(x))}= \mu'(x).
	\end{equation*}
\end{theorem}

\subsection{Convex curves}

We introduce some auxiliary notations and definitions, and we recall some observations and properties for convex curves in the plane. They guarantee a formalization of the definition of the affine arclength measure $\nu$ we gave in the Introduction. These properties are standard, but we were not able to find any clear reference for them. Therefore, for the sake of completeness we include the required proofs in Appendix~\ref{sec:appendix}.

\begin{definition}
	A set $K \subseteq \R^n$ is \emph{convex} if for all $x,y \in K$, $0 \leq \lambda \leq 1$ we have
	\begin{equation*}
	\lambda x + (1-\lambda) y \in K.
	\end{equation*}
	A \emph{convex curve} $\Gamma \subseteq \R^2$ is the boundary $\partial K$ of a non-empty open convex set $K \subseteq \R^2$.
\end{definition}
From now on, we restrict ourselves to \emph{compact convex curves}. We extend the definitions and results to every non-compact convex curve $\Gamma$ considering the sequence of compact convex curves 
\begin{equation*}
\{ \Gamma_N \coloneqq \partial ( K \cap [-N,N]^2 ) \colon N \in \N \} .
\end{equation*}
\begin{theorem} \label{thm:rectifiable}
	Every compact convex curve $\Gamma$ is rectifiable. 
\end{theorem}
Therefore, a compact convex curve $\Gamma$ admits an arclength parametrization 
\begin{equation*}
z \colon J = [0, \ell(\Gamma)) \to \Gamma \subseteq \R^2,
\end{equation*} 
where $\ell(\Gamma)$ is the length of the curve $\Gamma$. Without loss of generality, we assume the parametrization to be counterclockwise. Moreover, we have an almost identical arclength parametrization defined by
\begin{gather*}
\widetilde{z} \colon \widetilde{J} = (0, \ell(\Gamma)] \to \Gamma \subseteq \R^2, \\
\widetilde{z}(\ell(\Gamma)) \coloneqq  z(0), \qquad \qquad \forall t \in (0, \ell(\Gamma)), \widetilde{z}(t) \coloneqq  z(t).
\end{gather*}
With a slight abuse of notation, we denote both of the arclength parametrizations by $z$. The identification is harmless and involves a single point. At any time it will be made clear by the context which one is the appropriate choice of the arclength parametrization we are considering. A first instance of the feature just described appears in the following statement about the existence of well-defined left and right derivatives of the function $z$. Strictly speaking, we should define the left derivative $\widetilde{z}'_l$ of $\widetilde{z}$ on $\widetilde{J}$, and the right derivative $z'_r$ of $z$ on $J$.
\begin{theorem} \label{thm:derivative_z}
	The left and right derivatives $z'_l$ and $z'_r$ of $z$ with respect to the Lebesgue measure $m$ on $J$ are well-defined functions from $J$ to $\mathbb{S}^1$, and they coincide $m$-almost everywhere.
\end{theorem}
In fact, the functions $z'_l$ and $z'_r$ admit well-defined derivatives $m$-almost everywhere.
\begin{theorem} \label{thm:derivative_z'}
	The derivatives $z''_l$ and $z''_r$ of $z'_l$ and $z'_r$ with respect to the Lebesgue measure $m$ on $J$ are well-defined $m$-almost everywhere. They are functions from $J$ to $\R^2$ and coincide $m$-almost everywhere.
\end{theorem}
Next, we define the Borel measure $\sigma$ on $J$ as follows. For all $a,b \in J$, $a \leq b$ we define
\begin{equation} \label{eq:sigma}
\begin{aligned}
& \sigma((a,b)) \coloneqq \max \{ 0, \theta(z'_r(a),z'_l(b)) \}, \qquad \qquad & \sigma((a,b]) \coloneqq \theta(z'_r(a),z'_r(b)), \\
& \sigma([a,b)) \coloneqq \theta(z'_l(a),z'_l(b)), \qquad \qquad & \sigma([a,b]) \coloneqq \theta(z'_l(a),z'_r(b)).
\end{aligned}
\end{equation}
We denote by $\kappa$ the metric density associated with the absolutely continuous part of $\sigma$ with respect to the Lebesgue measure $m$ on $J$. 
\begin{theorem} \label{thm:kappa_zeta''}
	The measure $\sigma$ is positive. The function $\kappa$ coincides $m$-almost everywhere with the functions $\det \begin{pmatrix} z_l' & z_l'' \end{pmatrix}$ and $\det \begin{pmatrix} z_r' & z_r'' \end{pmatrix}$. 
\end{theorem}
Finally, we define the \emph{affine arclength measure $\nu$} on $J$ by
\begin{equation*}
\diff \nu(t) = \sqrt[3]{\kappa(t)} \diff t.
\end{equation*}
With a slight abuse of notation, we denote by $\nu$ also its push-forward to $\Gamma$ via the affine arclength parametrization $z$.

\section{Proof of Theorem~\ref{thm:comparability}} \label{sec:affine_measure}

We begin by stating and proving an auxiliary lemma about the qualitative relation between the affine measure $\mu$ and the Lebesgue measure $m$ on $J$.

\begin{lemma} \label{thm:abscont}
	The measure $\mu$ is absolutely continuous with respect to the Lebesgue measure $m$ on $J$, namely for every subset $E \subseteq J$ we have
	\begin{equation*}
	m(E) = 0 \Rightarrow \mu(E) = 0.
	\end{equation*}
\end{lemma}

In its proof, we need the following auxiliary definition.

\begin{definition} \label{def:rect_over_interval}
Let $I \subseteq J$ be an interval. Let $c$ and $d$ be in the closure $\bar{J}$ of $J$ such that $\bar{I}= [c,d]$. Assume that $\sigma((c,d)) \leq \pi/2$. We define the \emph{rectangle $R(I)$ over $I$} to be the minimal rectangle containing $z(I)$ as follows.

If $z'$ is constant on the interior of $I$, then $z(I)$ is a segment. The affine measure $\mu$ of $z(I)$ is zero, as $z(I)$ can be covered by arbitrarily thin rectangles. We define $R(I)$ to be the segment $z(I)$ itself.

If $z'$ is not constant on the interior of $I$, then we define $R(I)$ to be the rectangle with two adjacent vertices in $z(c)$ and $z(d)$, and minimal width $h(R(I))$, see Figure~\ref{fig:1}. The condition on $z'$ guarantees that $h(R(I)) > 0$. Moreover, let $b(R(I)) = \abs{z(d)-z(c)}$. Furthermore, let the point $z(e)$ be in the intersection between $z(I)$ and the side of the rectangle opposite to that connecting $z(c)$ to $z(d)$. Finally, let $\phi$ and $\psi$ be the angles defined by 
\begin{equation*}
\phi \coloneqq \theta( z(e) - z(c), z(d) - z(c) ), \qquad \qquad \psi \coloneqq \theta( z(c) - z(d), z(e) - z(d) ).
\end{equation*}
\end{definition}

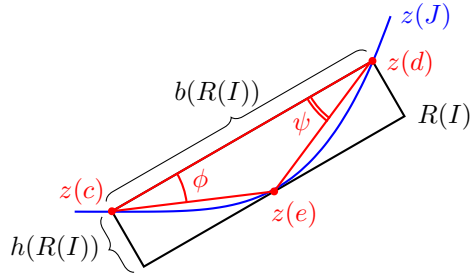
\begin{figure}[H]
	\centering
	\begin{tikzpicture}[scale=1]
	\draw[rotate around={30:(0,0)},thick] (0,0) rectangle (4,0.85);
	\draw (4,2)  node {$R(I)$};
	\draw[scale=1,domain=-0.5:3.7,smooth,variable=\x,blue,thick] plot ({\x-0.42},{\x*\x*\x*\x/72+0.73}) node[right] {$z(J)$};
	\draw (-0.4,1) node [above,red,left]{$z(c)$};
	\draw (2,1) node [red,below]{$z(e)$};
	\draw[red,thick,rotate around={30:(0,0)},thick] (0,0.85) -- (3.97,0.85) -- (2,0) -- (0,0.85);
	\draw (3.1,2.7) node [above,red, right]{$z(d)$};
	\draw[red] (0.75,1.15) node {$\phi$};
	\draw[red] (2.1,1.9) node {$\psi$};
	\draw[rotate around={30:(0,0)},red,thick] ([shift=(-23.5:1)]0,0.85) arc (-23.5:0:1);
	\draw[rotate around={30:(0,0)},red,thick] ([shift=(180:0.95)]4,0.85) arc (180:203.5:0.95);
	\draw[rotate around={30:(0,0)},red,thick] ([shift=(180:1)]4,0.85) arc (180:203.5:1);
	\fill[rotate around={30:(0,0)},red] (0,0.85) circle (1.5 pt);
	\fill[rotate around={30:(0,0)},red] (2,0) circle (1.5 pt);
	\fill[rotate around={30:(0,0)},red] (4,0.85) circle (1.5 pt);
	\draw [rotate around={30:(0,0)},decorate,decoration={brace,amplitude=4pt},xshift=0pt,yshift=0pt]
	(0,1) -- (4,1);
	\draw [rotate around={30:(0,0)},decorate,decoration={brace,amplitude=4pt},xshift=0pt,yshift=0pt]
	(-0.15,0) -- (-0.15,0.85);
	\draw[rotate around={30:(0,0)}] (-0.9,0.8) node {$h(R(I))$};
	\draw[rotate around={30:(0,0)}] (2,1.5) node {$b(R(I))$};
	\end{tikzpicture} 
	\caption{The rectangle $R(I)$ over the interval $I$.} \label{fig:1}
\end{figure}

\begin{proof}[Proof of Lemma~\ref{thm:abscont}]
Let $E \subseteq J$ be such that $m(E)=0$. We want to show that for every $\rho > 0$ there exists a covering of $z(E)$ by a collection of rectangles with bounded diameter such that the sum of their areas is bounded by $\rho$.

By assumption, $E$ has 1-dimensional Hausdorff measure zero. Therefore, for every $\varepsilon >0$ there exists a covering of $E$ by disjoint intervals $\{ I_n = [c_n,d_n) \colon n \in \N \}$ of bounded lengths $\ell_n=m(I_n) = \abs{d_n - c_n}$ such that
\begin{equation} \label{eq:measurezero}
\sum_{n \in \N} \ell_n \leq \varepsilon.
\end{equation}
Without loss of generality, up to splitting every interval into four subintervals, we can assume $\sigma(I_n) \leq \pi/2$.
 
The set $z(E)$ can be covered by the family $\{R_n \colon n \in \N \}$ of rectangles, where for every $n \in \N$ we define $R_n = R(I_n)$ to be the rectangle over the interval $I_n$ as in Definition~\ref{def:rect_over_interval}. The diameter of $R_n$ is bounded from above by
\begin{equation*}
\abs{z(e) - z(c)} + \abs{z(d) - z(e)}.
\end{equation*}
By the definition of the length of a curve, see Definition~\ref{def:rectifiability} in the Appendix, the sum in the previous display is bounded from above by $\ell(z(I_n))$. Finally, since $z$ is an arclength parametrization, we have that $\ell(z(I_n)) = \ell_n$. Therefore, for every $n \in \N$ the diameter of $R_n$ is bounded from above.

Moreover, we claim that for every $n \in \N$ we have
\begin{equation}  \label{eq:angle}
\frac{h_n}{\ell_n} \leq \sigma(I_n),
\end{equation}
where $h_n = h(R_n)$. In fact, for $e_n$, $\phi_n$, and $\psi_n$ as in Definition~\ref{def:rect_over_interval} and Figure~\ref{fig:1} we have
\begin{align*}
	\frac{h_n}{\ell_n} & \leq \frac{h_n}{\abs{c_n-e_n}} + \frac{h_n}{\abs{d_n-e_n}} \leq \\
	& \leq \frac{h_n}{\abs{z(c_n)-z(e_n)}} + \frac{h_n}{\abs{z(d_n)-z(e_n)}} = \sin \phi_n + \sin \psi_n \leq \phi_n + \psi_n \leq \sigma(I_n).
\end{align*}

Therefore, we have
\begin{align*}
\sum_{n \in \N} \abs{R_n}^{\frac{1}{3}} = \sum_{n \in \N} (b_n h_n)^{\frac{1}{3}} & \leq \sum_{n \in \N} \ell_n^{\frac{2}{3}} \Big(\frac{h_n}{\ell_n}\Big)^{\frac{1}{3}} \\ 
& \leq \sum_{n \in \N} \ell_n^{\frac{2}{3}} \sigma(I_n)^{\frac{1}{3}} \\
& \leq \Big( \sum_{n \in \N} \ell_n \Big)^{\frac{2}{3}} \Big( \sum_{n \in \N} \sigma(I_n) \Big)^{\frac{1}{3}} \\
& \leq \varepsilon^{\frac{2}{3}} (2 \pi)^{\frac{1}{3}} ,
\end{align*}
where we used the definition of the length of a curve to dominate $b_n = b(R_n)$ by $\ell_n$ in the first inequality, the inequality in \eqref{eq:angle} in the second, H\"{o}lder's inequality with the couple of exponents $(3/2, 3)$ in the third, and the inequality in \eqref{eq:measurezero}, the disjointness of $I_n$, and the definition of $\sigma$ in the fourth.   

By taking $\varepsilon$ arbitrarily small, we obtain the desired result.
\end{proof}

Next, we prove the quantitative relation between the affine measure $\mu$ and the affine arclength measure $\nu$ stated in Theorem~\ref{thm:comparability}.

\begin{proof} [Proof of Theorem~\ref{thm:comparability}]

Without loss of generality, up to splitting $J$ into eight disjoint subintervals, we can assume $\sigma(J) \leq \pi/4$. It is enough to prove the desired comparability for every subset $E \subseteq J$.

{\textbf{Part~I: $A \nu \leq \mu$.}} 
Let $R$ be a closed rectangle such that
\begin{equation*}
R \cap z(J) = z([c,d]),
\end{equation*}
where $[c,d] \subseteq J$. Let $\Phi \colon \Delta([c,d]) \to R+R$ be the function defined by
\begin{equation*}
\Phi (s,t) = z(s)+z(t).
\end{equation*}
The determinant of its Jacobian is defined $m$-almost everywhere, and it is
\begin{equation*}
\det \begin{pmatrix} z'(t) & z'(s) \end{pmatrix}.
\end{equation*}
Since the area of the subset $R+R$ is $4\abs{R}$, we have
\begin{equation} \label{eq:det_rect}
\begin{aligned}
4 \abs{R} & \geq  \int_{\Delta([c,d])} \det \begin{pmatrix} z'(t) & z'(s) \end{pmatrix} \diff s \diff t \\
& =  \int_{\Delta([c,d])} \Big( \int_{[t,s]}  \det \begin{pmatrix} z'(t) & \diff z' \end{pmatrix} \Big) \diff s \diff t \\
& \geq  \int_{\Delta([c,d])} \int_t^s  \det \begin{pmatrix} z'(t) & z''(u) \end{pmatrix} \diff u \diff s \diff t,
\end{aligned}
\end{equation}
where $\diff z'$ is the distributional derivative of $z'$, and $z''$ is a function coinciding $m$-almost everywhere with $z_l''$ and $z_r''$.

For $m$-almost all $t,u \in J$, $t \leq u$ we have
\begin{equation} \label{eq:k_det}
\begin{aligned}
 \det \begin{pmatrix} z'(t) & z''(u) \end{pmatrix} & = \abs{z''(u)} \sin( \theta( z'(t),z'(u)) + \theta( z'(u),z''(u)) )  \\
 & = \abs{z''(u)} \cos( \theta( z'(t),z'(u)) ) \\
 & \geq \frac{1}{2} \abs{z''(u)} \sin( \theta( z'(u),z''(u)) ) \\
 & = \frac{1}{2}  \det \begin{pmatrix} z'(u) & z''(u) \end{pmatrix},
\end{aligned}
\end{equation}
where in the second and in the third equality, as well as in the inequality we used
\begin{equation*}
\theta( z'(u),z''(u)) = \frac{\pi}{2},
\end{equation*}
and in the inequality we also used
\begin{equation*}
0 \leq \theta( z'(t),z'(u)) \leq \sigma(J) \leq \frac{\pi}{4}.
\end{equation*}

Therefore, there exists a constant $ C < \infty$ such that we have
\begin{align*}
\nu([c,d]) & = \int_c^d \sqrt[3]{\kappa(u)} \diff u \\
& = \int_c^d ((d-u)(u-c))^{-\frac{1}{3}} ((d-u)(u-c))^{\frac{1}{3}} \sqrt[3]{\kappa(u)} \diff u \\
& \leq \Big( \int_c^d ((d-u)(u-c))^{-\frac{1}{2}} \diff u \Big)^{\frac{2}{3}} \Big( \int_c^d (d-u)(u-c) \kappa(u) \diff u \Big)^{\frac{1}{3}} \\
& \leq C \Big( \int_c^d \int_c^u \int_u^d \kappa(u) \diff s \diff t \diff u \Big)^{\frac{1}{3}} \\
& \leq C \Big( \int_{\Delta([c,d])} \int_{s}^{t} \kappa(u) \diff u \diff s \diff t \Big)^{\frac{1}{3}} \\
& \leq 2 C \abs{R}^{\frac{1}{3}},
\end{align*}
where we used the definition of $\nu$ in the first equality, H\"{o}lder's inequality with the couple of exponents $(3/2,3)$ in the first inequality, we evaluated the first factor, which is independent of $c$ and $d$, in the third equality, we used Fubini in the second inequality, and we used Theorem~\ref{thm:kappa_zeta''} and the chains of inequalities in \eqref{eq:k_det} and \eqref{eq:det_rect} in the third inequality.  

Now, let $\{ R_n \colon n \in \N \}$ be a set of rectangles covering $z(E)$ and define $E_n \subseteq E$ by 
\begin{equation*}
z(E_n)= z(E) \cap R_n.
\end{equation*}
Then $\{ E_n \colon n \in \N \}$ is a covering of $E$, and we have
\begin{equation*}
\sum_{n \in \N} \abs{R_n}^{\frac{1}{3}} \geq 2 C \sum_{n \in \N} \nu(E_n) \geq 2 C \nu(E).
\end{equation*}
By taking the $\liminf$ over all the possible coverings, we obtain the desired inequality.

{\textbf{Part~II: $\mu \leq B \nu$.}} By Lemma~\ref{thm:abscont}, there exists a function $\mu' \colon J \to [0, \infty)$ defined $m$-almost everywhere such that for every measurable subset $E \subseteq J$ we have
\begin{equation*}
\mu(E) = \int_E \mu'(t) \diff t.
\end{equation*}

By Theorem~\ref{thm:derivative}, for $m$-almost every $t \in J$ we have
\begin{equation*}
\mu'(t)=\lim_{\varepsilon \to 0} \frac{ \mu([s,s+\varepsilon])}{\varepsilon}, \qquad \qquad \text{ where $t \in [s,s+\varepsilon]$. }
\end{equation*}

As in the proof of Lemma~\ref{thm:abscont}, the limit is bounded from above by
\begin{equation*}
\lim_{\varepsilon \to 0} \frac{ \varepsilon^{\frac{2}{3}} (\sigma ([s,s+\varepsilon]))^{\frac{1}{3}}}{\varepsilon} = \Big( \lim_{\varepsilon \to 0} \frac{ \sigma ([s,s+\varepsilon])}{\varepsilon} \Big)^{\frac{1}{3}}.
\end{equation*}

By Theorem~\ref{thm:derivative} and Theorem~\ref{thm:kappa_zeta''}, we obtain the desired inequality.
\end{proof}

\section{Proofs of Theorem~\ref{thm:uniform2} and the corollaries} \label{sec:proofs}

	We begin with an auxiliary definition.
	
	\begin{definition}	
	A measurable function $a$ in $\R^n$ is a \emph{bump function} if
	there exists a rectangular parallelepiped $R$ centred at the origin with sides parallel to the axes such that
	\begin{equation*}
	\norm{a}_{L^\infty(\R^n)}\le \abs{R}^{-1}1_R.
	\end{equation*}
	We denote by $\mathcal{A}_n$ the collection of bump functions on $\R^n$.
	\end{definition}
	The convolution with such bump functions is pointwise bounded by the strong Hardy-Littlewood maximal function, uniformly in the rectangle. 
	
	Next, we state and prove an auxiliary lemma about the boundedness properties of the adjoint operator of a certain linearised maximal Fourier restriction operator.
	
	\begin{lemma}\label{thm:lemma2}
		Let $1\le  r< 2$. There exists a constant $C=C(r) < \infty$ such that the following property holds true.
		
		For every convex curve $\Gamma$ parametrized by arclength $z: J \to \Gamma \subseteq \R^2$ and every collection $\{ a_{z(t)} \colon t \in J \} \subseteq \mathcal{A}_2$ of bump functions such that, as a function in $(t,x)$,
		\begin{equation*}
		a_{z(t)}(x) \in L^{\infty}(\diff \nu(t); L^1(\diff x)),
		\end{equation*}
		let $S = S(\Gamma,\{ a \})$ be the operator defined as follows. For every function $f \in L^4(J,\nu)$ we define
		\begin{equation*}
		Sf(\xi)=\int_{J} \widehat{a}_{z(t)}(\xi) e^{2\pi i  \xi \cdot z(t)}  f(t  ) \diff \nu (t) .
		\end{equation*}
		Then, we have
		\begin{equation*}
		\norm{Sf}_{L^{2r'}(\R^2)} \le C \norm{f}_{L^{\frac{2r}{3-r}}(J,\nu)} .
		\end{equation*}
	\end{lemma}
	
	Its proof relies on a lemma about the boundedness properties of an adjoint operator of a linearised maximal operator combined with a Fourier transform proved by Ramos.
	
	\begin{lemma} [Ramos \cite{MR4055940}, Lemma~1]\label{thm:lemma1}
		Let $n,k \geq 1$. There exists a constant $C=C(n,k) < \infty$ such that the following property holds true.
		
		For every collection 
		\begin{equation*}
		\Big\{ b_x \colon \widehat{b}_x = \prod_{i=1}^k \widehat{b}_{x,i} , b_{x,i} \in \mathcal{A}_n, x \in \R^n \Big\},
		\end{equation*}
		of convolution products of $k$ bump functions such that, as function in $(x,y)$, 
		\begin{equation*}
		{b_x}(y) \in L^\infty( \diff x; L^1( \diff y)),
		\end{equation*}
		let $T = T(\{b_x \})$ be the operator defined as follows. For every function $f \in L^2(\R^n) \cap L^1(\R^n)$ we define
		\begin{equation*}
		Tf(\xi)=\int_{\R^d} \widehat{b}_x(\xi) e^{2\pi i x  \cdot \xi}  f(x)\diff x .
		\end{equation*}
		Then, we have 
		\begin{equation*}
		\norm{ Tf }_{L^2(\R^n)} \leq C \norm{f}_{L^2(\R^n)} .
		\end{equation*}
	\end{lemma}
		
	\begin{proof}[Proof of Lemma~\ref{thm:lemma2}]
		Without loss of generality, by the definition of $\nu$, we restrict our attention to $I \subseteq J$ where $z'_l$ and $z'_r$ coincide, and $\kappa(t)$ is well-defined and strictly positive.
		
		Following the idea of Carleson-Sj\"{o}lin in \cite{MR361607} and Sj\"{o}lin in \cite{MR385437}, we rewrite the square of $Sf$ via a two-dimensional integral
		\begin{align*}
		Sf(\xi)^2 & =\int_{I\times I} \widehat{a}_{z(t)}(\xi) {\widehat{a}_{z(s)}(\xi)} e^{2\pi
			\xi \cdot (z(t)+z(s))}  f(t ) {f(s)}   \diff \nu(t) \diff \nu(s) \\
		& = 2 \int_{\Delta(I)} \widehat{a}_{z(t)}(\xi) {\widehat{a}_{z(s)}(\xi)} e^{2\pi
			\xi \cdot (z(t)+z(s))}  f(t ) {f(s)}   \diff \nu(t) \diff \nu(s). 
		\end{align*}
		We change variables via the bijective function $\Phi \colon \Delta(I) \to \Omega \subseteq \R^2$ defined by
		\begin{equation*}
			\Phi(s,t) = z(s)+z(t),
		\end{equation*}  
		and for $(s,t) \in \Delta(I)$ we define
		\begin{align*}
		\widehat{b}_{z(s)+z(t)} & \coloneqq  \widehat{a}_{z(s)}   {\widehat{a}_{z(t)}} , \\
		F(z(s)+z(t)) & \coloneqq  f(s)  {f(t)}\abs{\det\begin{pmatrix} z'(s) & z'(t) \end{pmatrix}}^{-1} \sqrt[3] {\kappa(t)} \sqrt[3] {\kappa(s)}.
		\end{align*}
		By the definition of $\nu$ and $\Phi$, we obtain
		\begin{equation*}
		Sf(\xi)^2= 2 \int_\Omega  \widehat{b}_{x}(\xi)e^{2\pi i \xi \cdot  x}  F(x)\diff x .
		\end{equation*}
		
		Next, we prove by interpolation that for every $1\le r\le 2$ there exists a constant $C=C(r) < \infty$ such that we have   
		\begin{equation*}
		\norm{Sf}_{L^{2r'}(\R^2)}^{2r}=\norm{Sf^2}_{L^{r'}(\R^2)}^r \le C \norm{F}_{L^r(\R^2)}^r .
		\end{equation*}
		The case $r=1$ follows from $\norm{\widehat{b}_x}_{L^\infty(\R^2)} \leq C$. The case $r=2$ follows from Lemma~\ref{thm:lemma1}. 
		
		After that, to estimate the $L^r(\R^2)$ norm of $F$ for $1\le r<2$, we invert the change of variables $\Phi$,
		\begin{equation} \label{eq:integral}
		\begin{aligned}
		\int_{\Omega} \abs{F(x)}^r dx & = \int_{\Delta(I)}    \abs{f(t  )f(s)}^{r} 
		{\kappa(t)^{\frac{r}{3}}} {\kappa(s)^{\frac{r}{3}}} \abs{\det\begin{pmatrix} z'(t) & z'(s) \end{pmatrix}}^{1-r} \diff t \diff s \\
		& = \int_{\Delta(I)} \abs{f(t  )f(s)}^{r} 
		{\kappa(t)^{\frac{r}{3}}} {\kappa(s)^{\frac{r}{3}}} \abs{\sin(\theta(t) -\theta(s))}^{1-r} \diff t \diff s,
		\end{aligned}
		\end{equation}
		where we define $\theta \colon I \to [0,2 \pi)$ by requiring
		\begin{equation*}
		\begin{pmatrix}	\cos \theta(t) \\ \sin \theta(t) \end{pmatrix} = z'(t).
		\end{equation*}
		
		We split $\Delta(I)$ in the four subsets defined as follows. For $j \in \{1,2,3,4\}$ we define
		\begin{equation*}
		\Delta_{j} \coloneqq  \Big\{ (s,t) \in \Delta(I) \colon \theta(s)-\theta(t) \in \Big[ \frac{(j-1) \pi}{2} , \frac{j \pi}{2} \Big)  \Big\} ,
		\end{equation*}
		and we observe that
		\begin{align*}
			& \text{ for $(s,t) \in \Delta_{1}$,} \qquad \qquad && \sin(\theta(s)-\theta(t)) \geq \frac{1}{2} (\theta(s)-\theta(t)) \geq 0, \\
			& \text{ for $(s,t) \in \Delta_{2}$,} \qquad \qquad && \sin(\theta(s)-\theta(t)) \geq \frac{1}{2} (\pi + \theta(t)-\theta(s)) \geq 0 , \\
			& \text{ for $(s,t) \in \Delta_{3}$,} \qquad \qquad && \sin(\theta(t)-\theta(s)) \geq \frac{1}{2} (\theta(s)-\theta(t)- \pi) \geq 0, \\
			& \text{ for $(s,t) \in \Delta_{4}$,} \qquad \qquad && \sin(\theta(t)-\theta(s)) \geq \frac{1}{2} (2 \pi + \theta(t)-\theta(s)) \geq 0.
		\end{align*}
		We obtain the desired estimate by controlling the portions of the integral in \eqref{eq:integral} in the corresponding subsets separately.
		
		{\bf{Case~I: $(s,t) \in \Delta_{1}$.}} We have
		\begin{equation} \label{eq:auxiliary_1}
		\theta(s)-\theta(t) \geq \int_t^s \kappa (u) \diff u \geq 0.
		\end{equation}
		By the assumption on $I$ made at the beginning of the proof, the function $\Psi_1 \colon \Delta_{1} \to \widetilde{\Delta}_{1} \subseteq [0, 2 \pi)^2$ defined by
		\begin{equation} \label{eq:bijection2}
		\begin{gathered}
		\Psi_1 (s,t) = (\alpha (s), \beta (t)) , \\
			\alpha(s)=\int_{0}^s  \kappa(u) \diff u , \qquad \qquad \beta(t)=\int_{0}^t  \kappa(u) \diff u,
		\end{gathered}
		\end{equation}
		is bijective. Together with the change of variables via the function $\Psi_1$, the inequality in \eqref{eq:auxiliary_1} yields that the portion of the integral in \eqref{eq:integral} on $\Delta_{1}$ is bounded from above by
		\begin{equation*}
		\int_{\widetilde{\Delta}_{1}}    \abs{f(s (\alpha) )}^{r}\abs{f(t(\beta))}^{r} 
		{\kappa(s(\alpha))^{\frac{r}{3}-1}} {\kappa(t(\beta))^{\frac{r}{3}-1}} \abs{\alpha- \beta}^{1-r} \diff \alpha\diff \beta .
		\end{equation*}
		By Hardy-Littlewood-Sobolev inequality, up to a multiplicative constant, the previous display is bounded from above by
		\begin{equation*}
		\norm{\abs{f \circ s}^r ( \kappa \circ s )^{\frac{r}{3}-1}}^2_{L^{\frac{2}{3-r}} (\widetilde{I}, \widetilde{\nu})},
		\end{equation*}
		where $\widetilde{I} = \alpha(I)$ and $\widetilde{\nu}$ is the push-forward to $\widetilde{I}$ via $\alpha$ of the measure $\nu$ on $I$. We change variables via the inverse of the bijective function $\alpha$ defined in \eqref{eq:bijection2}. Up to a multiplicative constant, we obtain the desired estimate for the portion of the integral in \eqref{eq:integral} on $\Delta_{1}$ by
		\begin{equation*}
		\norm{f}^{2r}_{L^{\frac{2r}{3-r}} (I,\nu)}.
		\end{equation*}
				
		{\bf{Case~II: $(s,t) \in \Delta_{2}$.}} For $S_2(t)$ defined by
		\begin{equation*}
		S_2(t) \coloneqq  	\sup \Big\{ u \in J \colon \theta(u) \leq \theta(t)+\pi    \Big\} \geq s,
		\end{equation*}
		we have
		\begin{equation} \label{eq:auxiliary_2}
		\pi + \theta(t)-\theta(s) \geq \int_s^{S_2(t)} \kappa (u) \diff u  \geq 0.
		\end{equation} 
		By the assumption on $I$ made at the beginning of the proof, the function $\Psi_2 \colon \Delta_{2} \to \widetilde{\Delta}_{2} \subseteq [0, 2 \pi)^2$ defined by
		\begin{equation} \label{eq:bijection3}
		\begin{gathered}
			\Psi_2 (s,t) = (\alpha (s,t), \beta (t)) , \\
			\alpha(s,t)=\int_s^{S_2(t)}  \kappa(u) \diff u+\int_{0}^t  \kappa(u) \diff u + \pi, \qquad \qquad \beta(t)=\int_{0}^t  \kappa(u) \diff u,
		\end{gathered}
		\end{equation}
		is bijective. Since the function $S_2$ is increasing then it is differentiable almost everywhere. Therefore, the function $\Psi_2$ is \emph{approximately totally differentiable} almost everywhere in its domain, see Theorem~1 and the following Example in \cite{MR1201446}. Together with the change of variables via the function $\Psi_2$, the inequality in \eqref{eq:auxiliary_2} yields that the portion of the integral in \eqref{eq:integral} on $\Delta_{2}$ is bounded from above by
		\begin{equation*}
		\int_{\widetilde{\Delta}_{2}}    \abs{f(s (\alpha) )}^{r}\abs{f(t(\beta))}^{r} 
		{\kappa(s(\alpha))^{\frac{r}{3}-1}} {\kappa(t(\beta))^{\frac{r}{3}-1}} \abs{\alpha- \beta- \pi}^{1-r} \diff \alpha\diff \beta ,
		\end{equation*}
		where we used the result stated in Theorem~2 in \cite{MR1201446} for changes of variables that are approximately totally differentiable almost everywhere.
		
		As in {\bf{Case~I}}, we conclude by Hardy-Littlewood-Sobolev inequality and the change of variables via the inverse of the bijective function defined in \eqref{eq:bijection3}.
		
		{\bf{Case~III: $(s,t) \in \Delta_{3}$.}} For $S_3(t)$ defined by
		\begin{equation*}
		S_3(t) \coloneqq  	\inf \Big\{ u \in J \colon \theta(u) \geq \theta(t)+\pi    \Big\} \leq s,
		\end{equation*}
		we have
		\begin{equation*}
		\theta(s) - \theta(t) - \pi \geq \int_{S_3(t)}^s \kappa (u) \diff u  \geq 0.
		\end{equation*} 
		We conclude as in {\bf{Case~II}}, with the change of variables via the bijective function $\Psi_3 \colon \Delta_{3} \to \widetilde{\Delta}_{3} \subseteq [0, 2 \pi)^2$ defined by
		\begin{equation*}
		\begin{gathered}
		\Psi_3 (s,t) = (\alpha (s,t), \beta (t)) , \\
		\alpha(s,t)=\int_{S_3(t)}^s  \kappa(u) \diff u+\int_{0}^t  \kappa(u) \diff u + \pi, \qquad \qquad \beta(t)=\int_{0}^t  \kappa(u) \diff u.
		\end{gathered}
		\end{equation*}
		
		{\bf{Case~IV: $(s,t) \in \Delta_{4}$.}} We have
		\begin{equation*}
		2 \pi + \theta(t) - \theta(s) \geq \int_s^{\ell(\Gamma)} \kappa (u) \diff u + \int_0^{t} \kappa (u) \diff u \geq 0.
		\end{equation*} 
		We conclude as in {\bf{Case~I}}, with the change of variables via the bijective function $\Psi_4 \colon \Delta_{4} \to \widetilde{\Delta}_{4} \subseteq [0, 2 \pi)^2$ defined by
		\begin{equation*}
		\begin{gathered}
		\Psi_4 (s,t) = (\alpha (s), \beta (t)) , \\
		\alpha(s) = 2 \pi - \int_s^{\ell(\Gamma)}  \kappa(u) \diff u, \qquad \qquad \beta(t)=\int_{0}^t  \kappa(u) \diff u.
		\end{gathered}
		\end{equation*}
	\end{proof}
	
	Next, we prove the boundedness of the maximal Fourier restriction operator uniformly in the convex curve stated in Theorem~\ref{thm:uniform2}.
	
\begin{proof} [Proof of Theorem~\ref{thm:uniform2}]
	The proof follows a standard argument that we repeat for the sake of completeness. Let $g \in L^\infty(\R^2)$ be a function normalized in $L^\infty(\R^2)$. Let $R$ be a measurable function associating a point in $\Gamma$ to a rectangle centred at the origin with sides parallel to the axes. We consider the linearised maximal Fourier restriction operator $\mathcal{M}_{g,R}$ defined as follows. For every Schwartz function $f \in \mathcal{S}(\R^2)$ we define 
\begin{equation*}
\mathcal{M}_{g,R} \widehat{f}(t)=\int_{\R^2} \widehat{f}(z(t) - y) g(z(t) - y) \abs{R(z(t))}^{-1} 1_{R(z(t))}(y) \diff y.
	\end{equation*}
	We aim at proving boundedness properties for $\mathcal{M}_{g,R}$ with constants independent of the linearising function $R$.
	
	The operator is bounded from $L^1(\R^2)$ to $L^\infty(J, \nu)$. To prove its boundedness properties near $L^{4/3}(\R^2)$, we introduce the bump function
	\begin{equation*}
	a_x(y) \coloneqq \abs{R(x)}^{-1} 1_{R(x)}(y)\overline{g(x-y)},
	\end{equation*}
	and, by Plancherel, we rewrite
	\begin{equation*}
	\mathcal{M}_{g,R} \widehat{f} (t) = \int_{\R^2} \overline{\widehat{a}_{z(t)} (\xi)} e^{2\pi i \xi \cdot
		z(t)}f(\xi) \diff \xi.
	\end{equation*} 
	 The adjoint operator $\mathcal{M}_{g,R}^\ast$ with respect to the $L^1(J, \nu)$-pairing is defined by
	\begin{equation*}
	\mathcal{M}_{g,R}^\ast h(\xi)=\int_{J} \widehat{a}_{z(t)} (\xi)e^{-2\pi i \xi \cdot z(t)}h(t) \diff  \nu(t).
	\end{equation*}
	By Lemma~\ref{thm:lemma2}, for $1 \leq r < 2$ we have
	\begin{equation*}
	\mathcal{M}_{g,R}^\ast \colon L^{\frac{2r}{3-r}}(J, \nu) \to L^{2r'}(\R^2), \qquad \qquad \norm{\mathcal{M}_{g,R}^\ast}_{\op} < \infty,
	\end{equation*}  
	hence, for $1 \leq p < 4/3$ we have the desired result
	\begin{equation*}
	\mathcal{M}_{g,R} \colon L^{p}(\R^2) \to L^{\frac{p'}{3}}(J, \nu), \qquad \qquad \norm{\mathcal{M}_{g,R}}_{\op} < \infty,
	\end{equation*}  
	where $\norm{\cdot}_{\op}$ stands for the norm of the operator and $p' = 2r'$. 
	\end{proof}

	Finally, we prove the corollaries.

\begin{proof} [Proof of Corollary~\ref{thm:uniform}]
	For every function $f \in \mathcal{S}(\R^2)$ we define the function $g$ by
	\begin{equation*}
		g(\xi) = 
		\begin{cases} 
			\displaystyle \frac{\abs{\widehat{f}(\xi)}}{\widehat{f}(\xi)}, \qquad \qquad &\text{if $\widehat{f}(\xi) \neq 0$,} \\
			1, \qquad \qquad &\text{if $\widehat{f}(\xi) =0$.}
			\end{cases}
		\end{equation*}
	In particular, we have 
	\begin{equation*}
	\norm{g}_\infty =1, \qquad \qquad \widehat{f} g = \abs{\widehat{f}}.
	\end{equation*}
	
	Therefore, the function $\mathcal{M}_g \widehat{f}$ dominates the function $\abs{\widehat{f}}$, and the desired result follows from Theorem~\ref{thm:uniform2}.
\end{proof}

\begin{proof} [Proof of Corollary~\ref{thm:lebsegue}]
	The desired result holds true for every function $f \in \mathcal{S}(\R^2)$. 
	
	For $1 \leq p < 4/3$, the desired result for every function $f \in L^p(\R^2)$ follows from a standard approximation argument and the boundedness properties of the maximal operator stated in Theorem~\ref{thm:uniform2}.
\end{proof}

\appendix

\section{Compact convex curves} \label{sec:appendix}

\subsection{Proof of Theorem~\ref{thm:rectifiable}}

First, for every compact convex curve $\Gamma$ we define a continuous parametrization $\gamma$ in Lemma~\ref{thm:gamma}. We achieve this formalizing the following intuition. Let $x_0$ be a point in the bounded open convex set $K \subseteq \R^2$, whose boundary $\partial K$ is $\Gamma$. We parametrize $\Gamma$ by $\mathbb{S}^1$ via the unique intersection between $\Gamma$ and each positive half-line emanating from $x_0$. Moreover, we choose to parametrize $\mathbb{S}^1$ by $[0,2\pi)$ counterclockwise, hence $\Gamma$ too.

\begin{figure}[H]
	\centering
	\begin{tikzpicture}[scale=1]
	\draw[blue, thick] (0.5,0.0) arc (0:360:0.5);
	\draw[blue] (0.5,0.5) node [above]{$\mathbb{S}^1$};
	\draw (0,0) node [below]{$x_0$};
	\draw[red] (1.5,0.75) node [below]{$\tau(e)$};
	\draw (-2,0) node [left] {$\Gamma = \partial K$};
	\draw [thick] plot [smooth cycle, tension=1] coordinates {(-1.5,-1) (2,-0.5) (1.5,1) (-1,1)};
	\draw[red,thick,->] (0,0) -- (2.5,5/3) node [right]{$e$};
	\fill[red] (1.5,1) circle (1.5pt);
	\fill (0,0) circle (1.5pt);
	\end{tikzpicture} 
	\caption{The intuitive parametrization of $\Gamma=\partial K$.}
\end{figure}
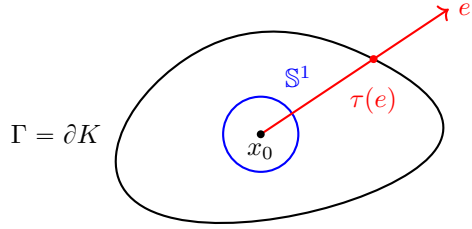

After that, we prove the rectifiability of every compact convex curve $\Gamma = \partial K$ claimed in Theorem~\ref{thm:rectifiable}. The main ingredient in the proof is the inequality between the perimeters of convex polygons $A$ and $B$ such that $A \subseteq B$ stated in Lemma~\ref{thm:polygons}. 

We begin with the definition of the continuous parametrization $\gamma$ for every compact convex curve $\Gamma = \partial K$ outlined above. We first state and prove three auxiliary lemmata.   

\begin{lemma} \label{thm:convexcombination}
	Let $x \in K$, $y \in \partial K$. For every $0 < \lambda \leq 1$ we have $\lambda x + (1-\lambda)y \in K$.
\end{lemma}
\begin{proof}
	Fix $0< \lambda \leq 1$. Since $y \in \partial K$, there exists a sequence $\{ y_n \colon n \in \N \} \subseteq K$ converging to $y$. Moreover, the sequence $\{ x_n \colon n \in \N \}$ defined by
	\begin{equation*}
	x_n \coloneqq  x - \frac{1-\lambda}{\lambda} (y - y_n),
	\end{equation*}
	converges to $x$. Therefore, there exists $N$ such that $x_N \in K$, yielding
	\begin{equation*}
	\lambda x +(1-\lambda) y = \lambda x_N + (1-\lambda) y_N \in K.
	\end{equation*}
\end{proof}

\begin{lemma} \label{thm:intersection}
	Let $x_0 \in K$. The function $T=T(x_0)$ defined by
	\begin{equation*}
	T \colon \mathbb{S}^1 \to (0, \infty), \qquad \qquad T(e) \coloneqq  \sup \Big\{ t \geq 0 \colon x_0+te \in K \Big\}, 
	\end{equation*}
	is well-defined. Moreover, for every $e \in \mathbb{S}^1$ we have
	\begin{equation*}
	\partial K \cap \{ x_0+te \colon t \geq 0 \} = \{ x_0+T(e)e \}.
	\end{equation*}
\end{lemma}
\begin{proof}
	Since $K$ is open and bounded, for every $e \in \mathbb{S}^1$ we have $T(e) \in (0,\infty)$. 
	
	Next, by the definition of $T(e)$, there exists an increasing sequence $\{t_n \colon n \in \N \} \subseteq (0,\infty)$ converging to $T(e)$. Therefore, the sequence $\{ x_0+t_n e \colon n \in \N \} \subseteq K$ converges to $x_0+T(e)e$. Since for $T(e)<t< \infty$ the point $x_0+te \in \R^2 \setminus K$, then $x_0+T(e)e \in \partial K$.

To conclude, suppose there exists $t >0$, $t \neq T(e)$ such that $x_0+te \in \partial K$. 

If $t< T(e)$, by Lemma~\ref{thm:convexcombination} we have $x_0+te \in K$, yielding a contradiction with $x_0+te \in \partial K$. 

If $t> T(e)$, the same argument yields a contradiction with $x_0 +T(e)e \in \partial K$.
\end{proof}
\begin{lemma} \label{thm:intersection_2}
Let $x_0 \in K$. For $T=T(x_0)$ the function $\tau = \tau(x_0)$ defined by
\begin{equation*}
\tau \colon \mathbb{S}^1 \to \partial K \subseteq \R^2, \qquad \qquad \tau(e) \coloneqq x_0+T(e)e,
\end{equation*}
is well-defined and bijective
\end{lemma}
\begin{proof}
The function is well-defined by Lemma~\ref{thm:intersection}.
	
{\bf{Injective.}} Suppose there exist $e_1,e_2 \in \mathbb{S}^1$, $e_1 \neq e_2$ such that
\begin{equation*}
x_0+T(e_1)e_1=x_0+T(e_2)e_2.
\end{equation*}

If $e_1 \neq -e_2$, they are two linearly independent vectors, hence $T(e_1)=T(e_2)=0$, yielding a contradiction with $T(e_1),T(e_2)>0$. 

If $e_1=-e_2$, then $T(e_1)=-T(e_2)$. Since $T(e_1)>0$, then $T(e_2)<0$, yielding a contradiction with $T(e_2)>0$.

{\bf{Surjective.}} Let $x \in \partial K$ and consider 
\begin{equation*}
e = \frac{x-x_0}{\abs{x-x_0}} \in \mathbb{S}^1.
\end{equation*}
Then $x \in \partial K \cap \{ x_0+te \colon t \geq 0 \}$. By Lemma~\ref{thm:intersection}, we have $x = x_0+T(e)e$.
\end{proof}

The remaining ingredient to define $\gamma$ is the following collection of parametrizations of $\mathbb{S}^1$. 

\begin{definition} \label{def:counterclockwise}
	Let $e \in \mathbb{S}^1 \subseteq \R^2$. We define the \emph{the counterclockwise continuous parametrization $\Theta = \Theta(e)$ of the circle $\mathbb{S}^1$ with starting point $e$} by
		\begin{equation*}
		\Theta \colon [0,2\pi) \to \mathbb{S}^1 \subseteq \R^2, \qquad \qquad \Theta(\theta) \coloneqq \begin{pmatrix}
		\cos \theta & -\sin \theta \\
		\sin \theta & \cos \theta 
		\end{pmatrix} e.
		\end{equation*} 
\end{definition}

In particular, for every $x_1 \in \partial K$ let $\Theta = \Theta(x_1)$ be the counterclockwise continuous parametrization of the circle $\mathbb{S}^1$ with starting point $\tau^{-1}(x_1) \in \mathbb{S}^1$.
\begin{lemma} \label{thm:gamma}
Let $x_0 \in K$, $x_1 \in \Gamma = \partial K$. For $\tau = \tau(x_0)$, $\Theta = \Theta(x_1)$ the function $\gamma = \gamma(x_0,x_1)$ defined by
\begin{equation*}
\gamma \colon [0,2 \pi) \to \Gamma = \partial K \subseteq \R^2, \qquad \qquad \gamma \coloneqq \tau \circ \Theta,
\end{equation*}
is well-defined, bijective and continuous. 
\end{lemma}
\begin{proof}

The function is well-defined and bijective by Lemma~\ref{thm:intersection_2} and the definition of $\Theta$. The continuity of $\gamma$ follows from that of $\Theta$ and $T \circ \Theta$.

It is enough to prove that the function $T \circ \Theta$ is continuous. We argue by contradiction and we suppose that it has a discontinuity in $\theta$. Let $\{\theta_n \colon n \in \N \}$ be a sequence converging to $\theta$ such that $\{  T(\Theta(\theta_n)) \colon n \in \N \}$ does not converge to $T(\Theta(\theta))$. In particular, there exists $\varepsilon > 0$ and a subsequence $\{ \theta_n \colon n \in M \subseteq \N \} \subseteq \{\theta_n \colon n \in \N \}$ such that 
\begin{equation*}
\inf \Big\{ \abs{ T(\Theta(\theta)) - T(\Theta(\theta_n))} \colon n \in M \Big\} \geq \varepsilon. 
\end{equation*}
Since $K$ is compact, there exists a subsequence $\{ \theta_n \colon n \in \widetilde{M} \subseteq M \} \subseteq \{\theta_n \colon n \in M \}$ such that the limit of $\{ T(\Theta(\theta_n)) \colon n \in \widetilde{M} \}$ exists and is $\widetilde{T} \neq T(\Theta(\theta))$. We distinguish two cases.

{\bf{Case~I: $\widetilde{T} > T(\Theta(\theta))$.}} Fix $t$ such that $\widetilde{T} > t > T(\Theta(\theta))$. The sequence 
\begin{equation*}
\Big\{ x_0+\frac{t}{\widetilde{T}} T(\Theta(\theta_n)) \Theta(\theta_n) \colon n \in \widetilde{M} \Big\} \subseteq K,
\end{equation*}
converges to $x_0+t \Theta(\theta)$. Therefore, we have $x_0+t \Theta(\theta) \in K \cup \partial K$. Then, by the convexity of $K$ and Lemma~\ref{thm:convexcombination}, we have $x_0+T(\Theta(\theta)) \Theta(\theta) \in K$, yielding a contradiction with $x_0+T(\Theta(\theta)) \Theta(\theta) \in \partial K$. 

{\bf{Case~II: $T(\Theta(\theta)) > \widetilde{T}$.}} Fix $t$ such that $ T(\Theta(\theta)) > t > \widetilde{T} $. The sequence 
\begin{equation*}
\Big\{ x_0+\frac{t}{\widetilde{T}} T(\Theta(\theta_n)) \Theta(\theta_n) \colon n \in \widetilde{M} \Big\} \subseteq \R^2\setminus ( K \cup \partial K),
\end{equation*}
converges to $x_0+t \Theta(\theta)$. Therefore, we have $x_0+t \Theta(\theta) \in \R^2 \setminus K$. However, by Lemma~\ref{thm:convexcombination}, we have $x_0+t \Theta(\theta) \in K$, yielding a contradiction.
\end{proof}

We continue with the proof that every compact convex curve $\Gamma = \partial K$ is rectifiable. We first recall the definition of rectifiability.

\begin{definition} \label{def:rectifiability}
Let $\gamma \colon I \to \Gamma \subseteq \R^2$ be a continuous parametrization of a curve, where $I \subseteq \R$ is a bounded interval of either of the following forms
\begin{equation*}
I = [a,b], \qquad I = [a,b), \qquad I = (a,b], \qquad I = (a,b).
\end{equation*}
Let $P=\{P_0,\dots,P_k\}$ be a finite and strictly increasing collection of points in $I$, namely $P_0 < P_1 < \dots < P_k$. Let $\sigma_{\gamma(P)}$ be the polygonal curve given by the segments between $\gamma(P_{i})$ and $\gamma(P_{i+1})$. Let $\ell(\sigma_{\gamma(P)})$ be the length of $\sigma_{\gamma(P)}$ defined by
\begin{equation*}
\ell(\sigma_{\gamma(P)}) \coloneqq \sum_{i=0}^{k-1} \abs{\gamma(P_{i+1})- \gamma(P_i)}.
\end{equation*}
Let $\mathcal{P}$ be the set of all possible finite and strictly increasing collections of points in $I$. The curve $\gamma(I)$ is \emph{rectifiable} if
\begin{equation*}
\ell(\gamma(I)) \coloneqq \sup \Big\{ \ell(\sigma_{\gamma(P)}) \colon P \in \mathcal{P} \Big\} < \infty,
\end{equation*}
and we call $\ell(\gamma(I))$ the \emph{length of $\gamma(I)$}.
\end{definition}

\begin{remark} \label{rmk:remark}
	If $I = [a,b]$, without loss of generality we consider only finite and strictly increasing collections $\{P_0, \dots, P_k\}$ of points in $I$ such that $P_0 = a$, $P_k = b$. 
\end{remark}
Now, for every parametrization $\gamma \colon [0, 2 \pi) \to \Gamma = \partial K$ we define the parametrization $\widetilde{\gamma} \colon [0, 2 \pi] \to \Gamma = \partial K$ by 
\begin{equation*}
\forall t \in [0, 2 \pi), \widetilde{\gamma}(t) \coloneqq  \gamma(t), \qquad \qquad \widetilde{\gamma}(2 \pi) \coloneqq  \gamma(0).
\end{equation*}
In particular, for $\widetilde{\gamma}$ we can apply the observation made in Remark~\ref{rmk:remark}. Moreover, it is straight-forward to observe that $\ell(\gamma([0,2\pi))) = \ell(\widetilde{\gamma}([0,2\pi]))$. Therefore, with a slight abuse of notation, we denote by $\gamma$ also $\widetilde{\gamma}$. 

Moreover, we introduce the auxiliary definition of convex hull we use in the remaining part of the Appendix.
\begin{definition}
Let $Q = \{ Q_1, \dots, Q_k \}$ be a finite collection of points in $\R^2$. The \emph{open convex hull $\ch( Q )$} is defined by
\begin{equation*}
\ch(Q) \coloneqq \Big\{ \sum_{i=1}^k \alpha_i Q_i \colon ( \alpha_1, \dots, \alpha_k ) \in (0,1)^k, \sum_{i=1}^k \alpha_i = 1 \Big\}.
\end{equation*}
\end{definition}

Next, we state and prove three auxiliary lemmata. 

\begin{lemma} \label{thm:points_ch}
Let $x,y \in \Gamma = \partial K$, $x \neq y$. Let $\gamma=\gamma(x) \colon [0,2 \pi] \to \Gamma = \partial K$ be the counterclockwise parametrization such that $\gamma(0)=x$. Let $s \in (0,2 \pi)$ be such that $\gamma(s)=y$. Then the two pieces $\gamma((0,s))$ and $\gamma((s,2\pi))$ of the curve $\Gamma$ are in the closure of the distinct half-planes defined by the line $l$ passing through $x$ and $y$.
\end{lemma}
\begin{proof}
Let $x_0 \in K$. Let $l_x$ be the half-line emanating from $x_0 $ and passing through $x$, and $l_y$ the half-line emanating from $x_0$ and passing through $y$. We distinguish three cases.

{\bf{Case~I: $s=\pi$.}} Then $l_x,l_y \subseteq l$, and the statement is satisfied.

{\bf{Case~II: $s<\pi$.}} In particular, $x_0 \notin l$. Let $H_0$ be the open half-plane such that $\partial H_0 = l$ and $x_0 \in H_0$. The piece $\gamma((0,s))$ of the curve $\Gamma$ is in the section of the plane defined by the counterclockwise angle from $l_x$ to $l_y$. We claim that $\gamma((0,s)) \subseteq H_0^c$. We argue by contradiction and we suppose that there exists $0<u<s$ such that $\gamma(u)$ belongs to the open subset $C = \ch( x,y,x_0 ) \subseteq K$. Then $\gamma(u) \in K$, yielding a contradiction with $\gamma(u) \in \partial K$.

Let $\Pi$ be the open section of the plane defined by the counterclockwise angle from $l_y$ to $l_x$. Let $A$ and $B$ be the connected open subsets of the plane such that $A \cap B = \varnothing$, $A \cup B = \Pi \cap ( H_0 \cup \partial H_0)^c$, $x \in \partial A$ and $y \in \partial B$. The piece $\gamma((s,2\pi))$ of the curve $\Gamma$ is in the set $\Pi$. We claim that $\gamma((s,2\pi)) \subseteq H_0 \cup \partial H_0$. We argue by contradiction and we suppose that there exists $s<u<2 \pi$ such that $\gamma(u)$ belongs to either of the subsets $A$ and $B$. Without loss of generality, we assume $\gamma(u) \in A$. Then $x \in \ch ( \gamma(u),y,x_0 ) \subseteq K$, yielding a contradiction with $x \in \partial K$.

{\bf{Case~III: $s >\pi$.}} We proceed as in {\bf{Case~II}}, switching the arguments for the two subcases.
\end{proof}

\begin{figure}[H]
	\centering
	\begin{tikzpicture}[scale=1]
	\draw[red,thick] plot [smooth, tension=0.5] coordinates {(3.5,0) (1.52,0.88) (0,1.2) (-1.13,0.660) (-2,-0.7)};
	\draw[thick,->] (0,0) -- (30:3.5cm) node [above]{$l_x$};
	\draw[thick,->] (0,0) -- (150:3cm) node [above]{$l_y$};
	\draw[thick] (-3,0.5) -- (3,1) node [above,right]{$l$};
	\draw[red,thick] (-0.17,0.1) arc (150:390:0.2);
	\draw[blue,thick] (0.17,0.1) arc (30:150:0.2);
	\draw[red] (5,0) node {$\Gamma=\partial K$};
	\draw (0,0.25) node [above]{$C$};
	\draw (-2.5,0.875) node {$B$};
	\draw (2.75,1.2) node {$A$};
	\draw (0.25,0) node [below,right]{$x_0$};
	\draw (-1,0.5) node [below]{$y$};
	\draw (1.5,0.75) node [below]{$x$};
	\fill (0,0) circle (1.5 pt);
	\fill (1.52,0.88) circle (1.5 pt);
	\fill (-1.13,0.660) circle (1.5 pt);
	\end{tikzpicture} 
	\caption{The open subsets $A,B,C$ in {\bf{Case~II}}.}
\end{figure}
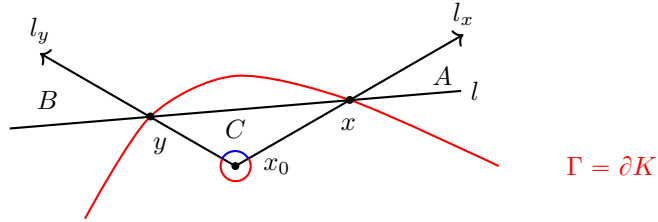

\begin{lemma} \label{thm:boundary_polygon}
	Let $\gamma \colon [0,2\pi] \to \Gamma \subseteq \R^2 $ be a parametrization of a compact convex curve $\Gamma$. Let $P = \{ P_0, \dots, P_k \}$ be a finite and strictly increasing collection of points in $[0,2\pi]$ such that $P_0 = 0$, $P_k = 2 \pi$. Then the open convex hull $\ch(\gamma(P))$ is an open convex polygon, and $\partial \ch(\gamma(P)) = \sigma_{\gamma(P)}$.
\end{lemma}
\begin{proof}
Consider the segment between $\gamma(P_j)$ and $\gamma(P_{j+1})$. By Lemma~\ref{thm:points_ch}, all the points in $\gamma(P)$ are in the same closed half-plane defined by the line passing through $\gamma(P_j)$ and $\gamma(P_{j+1})$. Therefore, the open convex hull $\ch(\gamma(P) )$ is in the same closed half-plane, and the segment between $\gamma(P_j)$ and $\gamma(P_{j+1})$ belongs to the boundary $\partial \ch( \gamma(P) )$.
\end{proof}
\begin{lemma} \label{thm:polygons}
Let $A,B$ be two convex polygons such that $A \subseteq B$. Then
\begin{equation*}
\ell(\partial A) \leq \ell(\partial B).
\end{equation*}
\end{lemma}
\begin{proof}
We prove the claim by induction on the number $n$ of sides of $\partial A$ that are not contained in $\partial B$. If $n=0$, then $A=B$ and the desired inequality is satisfied. 

Next, suppose that there are $n \geq 1$ sides of $\partial A$ that are not contained in $\partial B$. We choose one, we draw the line $l$ defined by it, and we let $H$ be the closed half-plane defined by $l$ containing $A$. Then $C= B \cap H$ is a convex polygon and, by triangle inequality, we have
\begin{equation*}
\ell(\partial C) \leq \ell(\partial B).
\end{equation*}
We observe that there are $n-1$ sides of $\partial A$ that are not contained in $\partial C$. Therefore, by induction hypothesis, we obtain the desired inequality.
\end{proof}
\begin{figure}[H]
	\centering
	\begin{tikzpicture}[scale=1]
	\draw [thick]  (-4,-1.2) -- (3,-1.2) -- (1,2.5) -- cycle;
	\draw [thick] (-1,-0.5) -- (-0.5,1) -- (1,0.75) -- (1.25,-0.75) -- cycle;
	\draw [red,thick] plot [smooth cycle, tension=0.8] coordinates {(-1,-0.5) (-0.5,1) (1,0.75) (1.25,-0.75)};
	\draw [thick,blue] (-3.5,1.5) -- (4,0.25) node [right,below] {$l$};
	\draw [blue] (-3,1) node [left,below] {$H$};
	\draw [red] (-1.2,-0.5) node [left] {$\Gamma=\partial K$};
	\draw (1,0) node [left] {$\partial A$};
	\draw (2,1.5) node [right] {$\partial B$};
	\end{tikzpicture} 
	\caption{The inductive step.}
\end{figure}
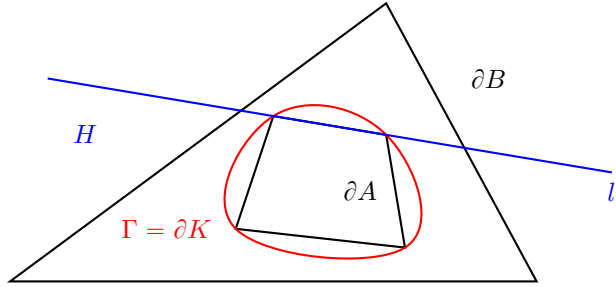

\begin{proof}[Proof of Theorem~\ref{thm:rectifiable}]
Let $B(0,R)$ be a ball centred at the origin with radius $R$ containing $K$. Let $\Delta$ be an equilateral triangle containing $B(0,R)$. 

By Lemma~\ref{thm:boundary_polygon}, for every finite and strictly increasing collection $P = \{P_0, \dots, P_k\}$ of points in $[0,2\pi]$ such that $P_0 = 0$, $P_k = 2 \pi$ the open convex hull $\ch( \gamma(P) )$ is an open convex polygon contained in $\Delta$. Moreover, we have $\sigma_{\gamma(P)}= \partial \ch ( \gamma(P) )$. 

By Lemma~\ref{thm:polygons}, we have
\begin{equation*}
\ell(\gamma(I)) \coloneqq \sup \Big\{ \ell(\sigma_{\gamma(P)}) \leq \ell(\partial \Delta) \colon P \in \mathcal{P} \Big\} < \infty. 
\end{equation*} 
\end{proof}

\begin{remark} \label{rmk:increasing_composition}
Let $x_0 \in K$, $x_1 \in \Gamma = \partial K$. Let $\gamma = \gamma(x_0, x_1) \colon [0,2\pi) \to \Gamma$ be the counterclockwise parametrization defined in Lemma~\ref{thm:gamma}. Let $z = z(x_1) \colon [0,\ell(\Gamma)) \to \Gamma$ be the counterclockwise affine arclength parametrization defined by
\begin{equation*}
z(0) = x_1.
\end{equation*}
The function $\gamma^{-1}\circ z$ is strictly increasing, because both $\gamma$ and $z$ are counterclockwise parametrizations.
\end{remark}

\subsection{Proofs of Theorem~\ref{thm:derivative_z}, Theorem~\ref{thm:derivative_z'}, and Theorem~\ref{thm:kappa_zeta''}}

We introduce two auxiliary functions $\theta_l$ and $\theta_r$ defined geometrically in every point of the convex curve $\Gamma = \partial K$ by the minimal cone centred at the point and containing the convex set $K$. These functions are strictly related to the left and right derivatives of the arclength parametrization $z$ of $\Gamma$, and are helpful in proving the desired theorems.

\begin{definition} \label{def:cone}
	Let $x$ be a point in $\Gamma = \partial K$. The \emph{cone $E_x$} is defined by
	\begin{equation*}
	E_x \coloneqq \Big\{ e \in \mathbb{S}^1 \colon \{ x+te \colon t > 0 \} \cap \partial K \neq \varnothing \Big\}.
	\end{equation*}
\end{definition}

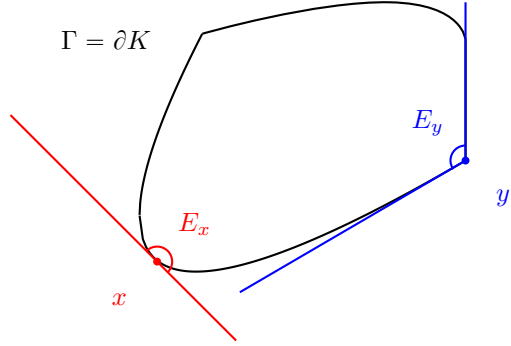
\begin{figure}[H]
	\centering
	\begin{tikzpicture}[scale=1]
	\draw[thick] plot[samples=100,domain=-1/3:0.5] (\x, {\x+1+sqrt(3*\x+1)});
	\draw[thick] plot[samples=100,domain=-1/3:4] (\x, {\x+1-sqrt(3*\x+1)});
	\draw[thick] (4,5-sqrt{13}) -- (4,3);
	\draw [thick] plot [smooth, tension=1] coordinates {(4,3) (3,3.5) (0.5,1.5+sqrt{2.5})};
	\draw[red,thick] (-2.05,2) -- (0.95,-1);
	\draw[blue,thick] (4,5-sqrt{13}) -- (4,3.5);
	\draw[blue,thick] (4,5-sqrt{13}) -- (1,2-8.5/sqrt{13});
	\draw (-1.5,3) node [right] {$\Gamma = \partial K$};
	\draw[red,thick] ([shift=(-45:0.2)]-0.1,0.05) arc (-45:135:0.2);
	\draw[blue,thick]([shift=(90:0.2)]4,5-sqrt{13}) arc (90:215:0.2); \draw[red] (-0.6,-0.45) node {$x$};
	\draw[blue] (4.5,4.5-sqrt{13}) node {$y$};
	\draw[red] (0.4,0.55) node {$E_x$};
	\draw[blue] (3.5,5.5-sqrt{13}) node {$E_y$};
	\fill[red] (-0.1,0.05) circle (1.5 pt);
	\fill[blue] (4,5-sqrt{13}) circle (1.5 pt);
	\end{tikzpicture} 
	\caption{Two instances of $E_x$.}
\end{figure}	

\begin{lemma} \label{thm:rlangle_1}
	For every $x \in \Gamma = \partial K$ we have $\mathbb{S}^1 \setminus E_x \neq \varnothing$. 
\end{lemma}

\begin{proof}	
	We argue by contradiction and we suppose that $E_x= \mathbb{S}^1$. We fix any arbitrary counterclockwise parametrization $\Psi \colon [0, 2 \pi) \to \mathbb{S}^1$ as in Definition~\ref{def:counterclockwise}. Let $y_1,y_2,y_3 \in \partial K$ be the points corresponding to the directions $e_1 = \Psi(\pi/3)$, $ e_2 = \Psi(\pi)$, and $e_3 = \Psi(5 \pi/3)$. Therefore, we have $x \in \ch(y_1,y_2,y_3)  \subseteq K$, yielding a contradiction with $x \in \partial K$.
\end{proof}

The previous result guarantees that the following definition is meaningful. For every $x \in \Gamma = \partial K$ let $e_0 = e_0(x) \in \mathbb{S}^1 \setminus E_x$. Moreover, let $\Phi = \Phi(e_0) \colon [0, 2 \pi) \to \mathbb{S}^1$ be the counterclockwise parametrization of the circle with starting point $e_0$ as in Definition~\ref{def:counterclockwise}. 

\begin{lemma} \label{thm:rlangle_2}
	For every $x \in \Gamma = \partial K$ we have that $\Phi^{-1}(E_x)$ is an interval with extremal points $a,b \in [0,2\pi)$ satisfying 
	\begin{equation} \label{eq:relation}
	a < b \leq a +\pi.
	\end{equation}
\end{lemma}
\begin{proof}
	Let $\theta_1, \theta_2 \in \Phi^{-1} (E_x)$ such that $\theta_1 < \theta_2$. We claim that for every $\theta \in [0, 2 \pi)$, $\theta_1 < \theta < \theta_2$ we have $e = \Phi(\theta) \in E_x$. 
	
	By the definition of $\Phi$, we have $\theta_1 \neq 0$ and $\theta_2 \neq 2 \pi$. Now, let $e_1, e_2 \in \mathbb{S}^1$ be defined by 
	\begin{equation*}
	e_1= \Phi(\theta_1), \qquad \qquad e_2= \Phi(\theta_2),
	\end{equation*}
	and let $y_1,y_2 \in \partial K$ be defined by
	\begin{equation*}
	y_1= \{ x+te_1 \colon t > 0 \} \cap \partial K, \qquad \qquad y_2= \{ x+te_2 \colon t > 0 \} \cap \partial K.
	\end{equation*}
	We distinguish three cases.
	
	{\bf{Case~I: $\theta_2>\theta_1+\pi$.}} We have
	\begin{equation*}
	\ch(y_1,y_2) \cap \{ x+te_0 \colon t > 0 \} \neq \varnothing,
	\end{equation*}
	yielding a contradiction with $e_0 \notin E_x$.
	
	{\bf{Case~II: $\theta_2<\theta_1+\pi$.}} We have
	\begin{equation*}
	\ch(y_1,y_2) \cap \{ x+te \colon t > 0 \} \neq \varnothing,
	\end{equation*}
	hence $\theta \in \Phi^{-1}(E)$.
	
	{\bf{Case~III: $\theta_2=\theta_1+\pi$.}} By {\bf{Case~I}}, we have 
	\begin{equation*}
	\theta_1 = \inf \Big\{ \theta \in \Phi^{-1}(E_x) \Big\}, \qquad \qquad \theta_2 = \sup \Big\{ \theta \in \Phi^{-1}(E_x) \Big\}.
	\end{equation*}
	Let $y \in K$. It belongs to one of the two half-planes defined by the line through $y_1,x,y_2$. Therefore, we have
	\begin{equation*} 
	\widetilde{\theta} \coloneqq \Phi^{-1} \Big( \frac{y-x}{\abs{y-x}} \Big) \in E_x, \qquad \qquad \theta_1 < \widetilde{\theta} < \theta_2 = \theta_1 + \pi,
	\end{equation*}
	and we reduce to {\bf{Case~II}} for the couples $(\theta_1,\widetilde{\theta})$ and $(\widetilde{\theta},\theta_2)$. 
		
	Therefore, $ \Phi^{-1}(E_x)$ is an interval with extremal points $a,b \in [0, 2 \pi)$. By {\bf{Case~I}}, we obtain the desired relation between $a,b$ described in \eqref{eq:relation}.
\end{proof}

In particular, $\Gamma = \partial K$ is contained in the closed section of the plane defined by the half-lines $\{ x+t \Phi(a) \colon t \geq 0  \}$ and $\{ x+t \Phi(b) \colon t \geq 0 \}$. Now, for every $x \in \Gamma = \partial K$ let $E_x$ be the cone as in Definition~\ref{def:cone} and let $e_0(x) \in \mathbb{S}^1 \setminus E_x$. Next, let $\Phi_{x} \colon [0, 2\pi) \to \mathbb{S}^1$ be the counterclockwise parametrization of the circle with starting point in $e_0(x)$ as in Definition~\ref{def:counterclockwise}. After that, let $x_1 \in \Gamma = \partial K$ and let the arclength parametrization $z = z(x_1) \colon J \to \Gamma$ be defined as in Remark~\ref{rmk:increasing_composition}. Then, we choose the counterclockwise parametrization of the circle $\Upsilon = \Upsilon(x_1) \colon [0,2\pi) \to \mathbb{S}^1$ with starting point 
\begin{equation*}
\Phi_{x_1} \Big( \inf \Big\{ \theta \in [0,2\pi) \colon \Phi_{x_1}(\theta) \in E_{x_1} \Big\} \Big).
\end{equation*}
as in Definition~\ref{def:counterclockwise}. Finally, we define the functions $\theta_l \colon (0,\ell(\Gamma)] \to [0, 2 \pi)$ and $\theta_r \colon [0,\ell(\Gamma)) \to [0,2 \pi)$ by
\begin{align*}
\theta_l (t) & \coloneqq \Upsilon^{-1} \Big( - \Phi_{z(t)} \Big( \sup \Big\{ \theta \colon \theta \in \Phi_{z(t)}^{-1}(E_{z(t)}) \Big\} \Big) \Big), \\
\theta_r (t) & \coloneqq \Upsilon^{-1} \Big( \Phi_{z(t)} \Big( \inf \Big\{ \theta \colon \theta \in \Phi_{z(t)}^{-1}(E_{z(t)}) \Big\} \Big) \Big).
\end{align*}

\begin{lemma} \label{thm:increasing}
	For all $s,t \in (0,\ell(\Gamma))$, $s < t$ we have
	\begin{equation} \label{eq:claim_1}
	\theta_r(s) \leq \theta_l(t) \leq \theta_r(t).
	\end{equation}
	Moreover, for every $s \in (0,\ell(\Gamma))$ we have
	\begin{equation} \label{eq:claim_2}
	\theta_r(0) \leq \theta_l(s) \leq \theta_r(s) \leq \theta_l(\ell(\Gamma)).
	\end{equation}
\end{lemma}
\begin{proof}
	The first inequality in \eqref{eq:claim_1} follows from
	\begin{equation} \label{eq:auxiliary_3}
	\theta_r(s) \leq \Upsilon^{-1} \Big( \frac{z(t)-z(s)}{\abs{z(t)-z(s)}} \Big) =  \Upsilon^{-1} \Big( - \frac{z(s)-z(t)}{\abs{z(s)-z(t)}} \Big) \leq \theta_l(t).
	\end{equation}
	The second inequality in \eqref{eq:claim_1} follows from Lemma~\ref{thm:rlangle_2} and the definition of a counterclockwise parametrization of $\mathbb{S}^1$ in Definition~\ref{def:counterclockwise}. The first and the third inequalities in \eqref{eq:claim_2} follow from the chain of inequalities in \eqref{eq:auxiliary_3}.
\end{proof}
	
\begin{lemma} \label{thm:coincidence}
	The functions $\theta_l$ and $\theta_r$ are increasing and have bounded variation. Moreover, they coincide $m$-almost everywhere.
\end{lemma}	
\begin{proof}
	By Lemma~\ref{thm:increasing}, the functions $\theta_l$ and $\theta_r$ are increasing. Moreover, they take values in a bounded set, hence they have bounded variation. 
	
	Now, suppose that the functions $\theta_l$ and $\theta_r$ do not coincide $m$-almost everywhere. Therefore, there exists an uncountable collection $X \subseteq (0,\ell(\Gamma))$ of points such that for every $x \in X$ we have
	\begin{equation*}
	\lim_{t \to x^{-}} \theta_r(t) \leq \theta_l(x) < \theta_r(x) \leq \lim_{t \to x^{+}} \theta_r(t).
	\end{equation*}
	Hence, we have
	\begin{equation*}
	\lim_{t \to \ell(\Gamma)^{-}} \theta_r(t) \geq \sum_{x \in X} \Big( \lim_{t \to x^{+}} \theta_r(t) - \lim_{t \to x^{-}} \theta_r(t) \Big) = \infty,
	\end{equation*}
	yielding a contradiction with $\theta_r([0,\ell(\Gamma)) \subseteq [0,2\pi)$. 
\end{proof}

\begin{lemma} \label{thm:properties}
	Fix $s \in J$ and consider the function $\phi = \phi_s$ defined by
	\begin{equation*}
	\phi \colon J \setminus \{s\} \to [0,2 \pi), \qquad \qquad \phi(t) \coloneqq \begin{cases} \displaystyle
	\Upsilon^{-1} \Big( \frac{z(s)-z(t)}{\abs{z(s)-z(t)}} \Big), \qquad \qquad & \text{if $t < s$,} \\
	\displaystyle
	\Upsilon^{-1} \Big( \frac{z(t)-z(s)}{\abs{z(t)-z(s)}} \Big), \qquad \qquad & \text{if $t > s$.}
	\end{cases}
	\end{equation*} 
	Then, the function $\phi$ is increasing.
\end{lemma}
\begin{proof}
	For all $t, u \in J \setminus \{ s\}$, $t < u$ we claim that 
	\begin{equation} \label{eq:claim_3}
	\phi(t) \leq \phi(u).
	\end{equation}
	
	Let $x_0 \in K$, $x_1 \in \Gamma = \partial K$, and let $\gamma = \gamma(x_0,x_1)$ and $z = z(x_1)$ be the associated parametrizations as in Remark~\ref{rmk:increasing_composition}. Moreover, we consider the points $z(s)$, $z(t)$, and $z(u)$. By Remark~\ref{rmk:increasing_composition}, we have 
	\begin{equation*}
	\gamma^{-1} (z(t)) < \gamma^{-1}(z(u)).
	\end{equation*}
	We distinguish three cases according to the relation between $s$, $t$, and $u$.
	
	{\bf{Case~I: $s < t < u$.}} We distinguish five additional subcases.
	
	{\bf{Case~I.i.}} We assume
	\begin{equation*}
	\gamma^{-1}(z(t)) < \gamma^{-1}(z(s))+\pi < \gamma^{-1}(z(u)).
	\end{equation*}
	By Lemma~\ref{thm:rlangle_2}, the points $z(t)$ and $z(u)$ belong to distinct open half-planes defined by the line passing through $z(s)$ and $\gamma(\gamma^{-1}(z(s))+\pi)$. Moreover, let $e_0 \in \mathbb{S}^1$ be defined by
	\begin{equation*}
	e_0 = \frac{z(s) - x_0}{\abs{z(s) - x_0}}.
	\end{equation*}
	In particular, we have
	\begin{equation*}
	- e_0 = \frac{\gamma(\gamma^{-1}(z(s))+\pi) - x_0}{\abs{\gamma(\gamma^{-1}(z(s))+\pi) - x_0}} \in E_{z(s)}.
	\end{equation*}
	By Lemma~\ref{thm:rlangle_2}, we have that $-e_0$  belongs to the interior of $E_{z(s)}$, hence we have $e_0 \in \mathbb{S}^1 \setminus E_{z(s)}$. Let $\Phi_{z(s)} \colon [0,2\pi) \to \mathbb{S}^1$ be the counterclockwise parametrization of the circle with starting point in $e_0$ as in Definition~\ref{def:counterclockwise}. To prove the desired inequality in \eqref{eq:claim_3}, it is enough to prove the inequality
	\begin{equation} \label{eq:claim_4}
	\Phi_{z(s)}^{-1} \Big( \frac{z(t)-z(s)}{\abs{z(t) - z(s)}} \Big) \leq \Phi_{z(s)}^{-1} \Big( \frac{z(u)-z(s)}{\abs{z(u) - z(s)}} \Big).
	\end{equation}
	To prove the desired inequality in \eqref{eq:claim_4}, we argue by contradiction and we suppose that 
	\begin{equation*}
	\Phi_{z(s)}^{-1} \Big( \frac{z(u)-z(s)}{\abs{z(u) - z(s)}} \Big) < \Phi_{z(s)}^{-1} \Big( \frac{z(t)-z(s)}{\abs{z(t) - z(s)}} \Big) .
	\end{equation*}
	Therefore, the points $z(t)$ and $z(u)$ belong to the same open half-plane defined by the line passing through $z(s)$ and $\gamma(\gamma^{-1}(z(s))+\pi)$, yielding a contradiction.

	{\bf{Case~I.ii.}} We assume
	\begin{equation*}
	\gamma^{-1}(z(t)) < \gamma^{-1}(z(u)) < \gamma^{-1}(z(s))+\pi .
	\end{equation*}
	Let $l_s$ and $l_u$ be the half-lines emanating from $x_0$ and passing through $z(s)$ and $z(u)$ respectively. Since the parametrization $z$ is counterclockwise, the point $z(t)$ belongs to the open section of the plane defined by the angle strictly smaller than $\pi$ between $l_s$ and $l_u$. To prove the desired inequality in \eqref{eq:claim_3}, we argue by contradiction and we suppose that $\phi(u) <  \phi(t)$. Let $l'_u$ and $l'_s$ be the half-lines emanating from $z(s)$ and passing through $z(u)$ and $x_0$ respectively. Since $\phi(u) <  \phi(t)$, the point $z(t)$ belongs to the open section of the plane defined by the angle strictly smaller than $\pi$ between $l'_u$ and $l'_s$. Therefore, we obtain
	\begin{equation*}
	z(t) \in \ch ( z(s),z(u),x_0 ) \subseteq K,
	\end{equation*}
	yielding a contradiction with $z(t) \in \Gamma = \partial K$.

	{\bf{Case~I.iii.}} We assume
	\begin{equation*}
	\gamma^{-1}(z(s))+\pi < \gamma^{-1}(z(t)) < \gamma^{-1}(z(u)) .
	\end{equation*}
	We argue by contradiction and we suppose that $\phi(u) <  \phi(t)$. Analogously to {\bf{Case~I.ii}}, we obtain 
	\begin{equation*}
	z(u) \in \ch ( z(s),z(t),x_0 ) \subseteq K,
	\end{equation*}
	yielding a contradiction with $z(u) \in \Gamma = \partial K$.
	
	{\bf{Case~I.iv.}} We assume
	\begin{equation*}
	\gamma^{-1}(z(t)) < \gamma^{-1}(z(s))+\pi = \gamma^{-1}(z(u)).
	\end{equation*}
	The desired inequality in \eqref{eq:claim_3} follows from the fact that the parametrizations $z$, $\gamma$, and $\Upsilon$ are counterclockwise.
	
	{\bf{Case~I.v.}} We assume
	\begin{equation*}
	\gamma^{-1}(z(t)) = \gamma^{-1}(z(s))+\pi < \gamma^{-1}(z(u)).
	\end{equation*}
	We prove the desired inequality in \eqref{eq:claim_3} analogously to {\bf{Case~I.iv}}.

	{\bf{Case~II: $	t < u < s$.}} We distinguish five additional subcases and we prove the desired inequality in \eqref{eq:claim_3} analogously to {\bf{Case~I}}.

	{\bf{Case~III: $t < s < u$.}} We prove the desired inequality in \eqref{eq:claim_3} by {\bf{Case~I}} applied to $\phi_t$ and {\bf{Case~II}} applied to $\phi_u$.	
\end{proof}

\begin{lemma} \label{thm:continuity}
	The function $\theta_r$ is right-continuous, and the function $\theta_l$ is left-continuous.
\end{lemma}	
\begin{proof}
We focus on the case of the function $\theta_r$. The case of the function $\theta_l$ is analogous. 

We want to prove that for every $s \in J$ we have
\begin{equation*}
\theta_r(s) = \lim_{t \to s^{+}} \theta_r (t). 
\end{equation*}
We fix $s \in J$. By Lemma~\ref{thm:coincidence}, the limit is an infimum and it is enough to prove that for every $\varepsilon > 0$ there exists $t > s$ such that
\begin{equation*}
\theta_r(t) \leq \theta_r (s) + 2 \varepsilon. 
\end{equation*}
By the definition of $\theta_r$, there exists $u \in J$, $u > s$ such that
\begin{equation} \label{eq:auxiliary_6}
\theta_r(s) \leq \Upsilon^{-1}\Big(\frac{z(u)-z(s)}{\abs{z(u)-z(s)}}\Big) \leq \theta_r(s)+ \varepsilon.
\end{equation}
By Lemma~\ref{thm:points_ch}, the piece $z((s,u))$ of the curve $\Gamma$ is in the closure of the half-plane defined by the line passing through $z(s)$ and $z(u)$. In particular, by the definition of $\theta_r$ and $\theta_l$, and Lemma~\ref{thm:properties}, for every $t \in J$, $s < t < u$ we have
\begin{equation*}
\theta_r(s) \leq \Upsilon^{-1} \Big( \frac{z(t) - z(s)}{\abs{z(t) - z(s)}} \Big)
\leq \Upsilon^{-1} \Big( \frac{z(u) - z(t)}{\abs{z(u) - z(t)}} \Big) \leq \theta_l (u).
\end{equation*}
We distinguish two cases.

{\bf{Case~I.}} We assume
\begin{equation*}
\Upsilon^{-1}\Big(\frac{z(u)-z(s)}{\abs{z(u)-z(s)}}\Big) = \theta_r(s).
\end{equation*}
Then, we have $\theta_l(u) = \theta_r(s)$. By Lemma~\ref{thm:increasing}, for every $t \in J$, $s < t < u$ we have 
\begin{equation*}
\theta_r(t) = \theta_r(s).
\end{equation*} 

{\bf{Case~II.}} We assume
\begin{equation*}
\Upsilon^{-1}\Big(\frac{z(u)-z(s)}{\abs{z(u)-z(s)}}\Big) > \theta_r(s).
\end{equation*}
Then, we have $\theta_l(u) > \theta_r(s)$. By Lemma~\ref{thm:rlangle_2}, there exists $t \in J$, $s < t < u$ such that
\begin{equation} \label{eq:auxiliary_7}
\begin{aligned}
0 \leq \Upsilon^{-1}\Big(\frac{z(u)-z(t)}{\abs{z(u)-z(t)}}\Big) -  \Upsilon^{-1}\Big(\frac{z(u)-z(s)}{\abs{z(u)-z(s)}}\Big) & =  \\
= \Phi_{z(u)}^{-1}\Big(\frac{z(t)-z(u)}{\abs{z(t)-z(u)}}\Big) - & \Phi_{z(u)}^{-1}\Big(\frac{z(s)-z(u)}{\abs{z(s)-z(u)}}\Big) \leq \varepsilon.
\end{aligned}
\end{equation}
Together with the definition of $\theta_r$, the inequalities in \eqref{eq:auxiliary_6} and \eqref{eq:auxiliary_7} yield
\begin{equation*}
\theta_r (t) \leq \Upsilon^{-1}\Big(\frac{z(u)-z(t)}{\abs{z(u)-z(t)}}\Big) \leq \theta_r(s) + 2\varepsilon. 
\end{equation*}
\end{proof}

We turn now to the derivatives $z_l'$ and $z_r'$ and their relation with the functions $\theta_l$ and $\theta_r$.

\begin{proof}[Proof of Theorem~\ref{thm:derivative_z}]

In the proof that the derivatives are well-defined, we focus on the case of the right derivative $z'_r$. The case of the left derivative $z'_l$ is analogous. 

We want to prove that for every $s \in J$ the limit
\begin{equation*}
z'_r(s) \coloneqq \lim_{t \to s^{+}} \frac{z(t)-z(s)}{t-s} ,
\end{equation*}
is well-defined in $\mathbb{S}^1$. 

We fix $s \in J$, we choose $\varepsilon>0$ such that $ s + \varepsilon \in J$. First, we consider the function $\psi = \psi(s)$ defined by
\begin{equation*}
\psi \colon [s, s+ \varepsilon) \to [0,2 \pi), \qquad \qquad \psi(t) \coloneqq \Upsilon^{-1} \Big( \frac{z(t)-z(s)}{\abs{z(t)-z(s)}} \Big) - \theta_r(s).
\end{equation*} 
By the definition of $\theta_r$ and Lemma~\ref{thm:properties}, the following limit exists and we have
\begin{equation*}
\lim_{t \to s^{+}} \Upsilon^{-1} \Big( \frac{z(t)-z(s)}{\abs{z(t)-z(s)}} \Big) \geq \theta_r(s).
\end{equation*}
Moreover, by the definition of $\theta_r$, for every $\delta>0$ there exists $t \in J$, $t > s$ such that
\begin{equation*}
\Upsilon^{-1}\Big(\frac{z(t)-z(s)}{\abs{z(t)-z(s)}}\Big) \leq \theta_r(s)+\delta.
\end{equation*}
Therefore, by Lemma~\ref{thm:properties} we have
\begin{equation} \label{eq:limit}
\lim_{t \to s^{+}} \Upsilon^{-1}\Big(\frac{z(t)-z(s)}{\abs{z(t)-z(s)}}\Big) = \theta_r(s).
\end{equation}
To conclude that $z_r'$ is well-defined in $\mathbb{S}^1$, it is enough to prove that 
\begin{equation} \label{eq:auxiliary_5}
\lim_{t \to s^{+}} \frac{\abs{z(t)-z(s)}}{t-s} = 1.
\end{equation}

By Lemma~\ref{thm:continuity}, for $\rho > 0$ small enough we have
\begin{equation} \label{eq:condition}
\theta_l(s+\rho) \leq \theta_r(s+\rho) < \theta_r(s)+\frac{\pi}{2}.
\end{equation}
For every $t \in (s,s+\rho)$ let $y(s,t)$ be the intersection between the half-line emanating from $z(s)$ in the direction $\Upsilon^{-1}(\theta_r(s))$ and the half-line emanating from $z(t)$ in the direction $- \Upsilon^{-1}(\theta_l(t))$. By the inequalities in \eqref{eq:condition}, the arc $z([s,t])$ of the curve $\Gamma$ is contained in the closure of the open convex hull $\ch(z(s),z(t),y(s,t))$, which is an obtuse triangle. This obtuse triangle is contained in a right-triangle with the segment between $z(s)$ and $z(t)$ as hypotenuse and a cathetus on the half-line emanating from $z(s)$ in the direction $\Upsilon^{-1}(\theta_r(s))$. By an argument analogous to that used to prove Theorem~\ref{thm:rectifiable}, we have
\begin{equation*}
t-s \leq \abs{z(t) - y(s,t)} + \abs{y(s,t) - z(s)} \leq (\sin \psi(t)+ \cos \psi(t) ) \abs{z(t)-z(s)},
\end{equation*}
where $t-s$ is the length of the arc $z([s,t])$ of the curve $\Gamma$. Therefore, we have
\begin{equation*}
\lim_{t \to s^{+}} \frac{\abs{z(t)-z(s)}}{t-s} \geq \lim_{t \to s^{+}} \frac{1}{\sin \psi(t)+ \cos \psi(t)} = 1.
\end{equation*}
Together with $\abs{z(t)-z(s)} \leq t-s$, the inequality in the previous display yields the desired equality in \eqref{eq:auxiliary_5}. 

In particular, by the equality in \eqref{eq:limit}, we proved
\begin{equation} \label{eq:auxiliary_4} 
z'_l = \Upsilon \circ \theta_l, \qquad \qquad z'_r = \Upsilon \circ \theta_r,
\end{equation}
Therefore, by Lemma~\ref{thm:coincidence}, the functions $z'_l$ and $z'_r$ coincide $m$-almost everywhere.
\end{proof}

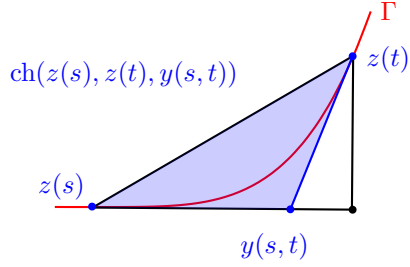
\begin{figure}[H]
	\centering
	\begin{tikzpicture}[scale=1]
	\draw[scale=1,domain=-0.5:3.7,smooth,variable=\x,red,thick] plot ({\x-0.42},{\x*\x*\x*\x/72+0.73}) node[right] {$\Gamma$};
	\draw (-0.4,1) node [above,blue,left]{$z(s)$};
	\draw (2,0.5) node [blue,below]{$y(s,t)$};
	\draw (3.1,2.7) node [above,blue, right]{$z(t)$};
	\draw (0,2.5) node [blue]{$\ch(z(s),z(t),y(s,t))$};
	\draw[rotate around={30:(0,0)},thick] (0,0.835) -- (2.95,-0.9);
	\draw[rotate around={120:(0,0)},thick] ({0.6*(2-3.5)},{0.6*(2.835-7.8)}) -- ({0.6*(4.95-3.5)},{0.6*(1.1-7.8)});
	\draw[rotate around={30:(0,0)},thick] (4,0.85) -- (0,0.835);
	\draw[rotate around={30:(0,0)},blue,thick] (4,0.85) -- (2.265,-0.50);
	\fill[rotate around={30:(0,0)},blue,opacity=0.2] (0,0.835) -- (2.265,-0.50) -- (4,0.85) -- (0,0.835);
	
	\fill[rotate around={30:(0,0)},blue] (0,0.85) circle (1.5 pt);
	\fill[rotate around={30:(0,0)},blue] (2.265,-0.50) circle (1.5 pt);
	\fill[rotate around={30:(0,0)},blue] (4,0.85) circle (1.5 pt);
	\fill[rotate around={120:(0,0)}] ({0.6*(2-3.53)},{0.6*(2.835-7.8)}) circle (1.5 pt);
	\end{tikzpicture}  
	\caption{The obtuse triangle $\ch(z(s),z(t),y(s,t))$ shaded in blue and the associated right-triangle in black.}
\end{figure}
Finally, we recall a result about the differentiability of a function of bounded variation.

\begin{theorem} [Stein and Shakarchi \cite{MR2129625}, Theorem 3.4] \label{thm:derivative_increasing}
	Let $a,b \in \R$. If $F$ is of bounded variation on $[a,b]$, then $F$ is differentiable almost everywhere.
\end{theorem}

\begin{proof} [Proof of Theorem~\ref{thm:derivative_z'}]
	By Lemma~\ref{thm:coincidence}, the functions $\theta_l$ and $\theta_r$ have bounded variation. By Theorem~\ref{thm:derivative_increasing}, they admit derivatives $\theta_l'$ and $\theta_r'$ well-defined $m$-almost everywhere. 
	
	Moreover, by Lemma~\ref{thm:increasing}, the function $\theta_r - \theta_l$ is positive everywhere. By Lemma~\ref{thm:coincidence}, it has bounded variation and it is zero $m$-almost everywhere. By Theorem~\ref{thm:derivative_increasing}, it admits a derivative $m$-almost everywhere, hence the derivative is zero $m$-almost everywhere. Therefore, the functions $\theta'_l$ and $\theta'_r$ coincide $m$-almost everywhere.
		
	As we concluded in \eqref{eq:auxiliary_4}, we have 
	\begin{equation} \label{eq:determinants_1}
	z_l'(t) = \begin{pmatrix} \cos \theta_l(t) \\ \sin \theta_l(t) \end{pmatrix} , \qquad \qquad
	z_r'(t) = \begin{pmatrix} \cos \theta_r(t) \\ \sin \theta_r(t) \end{pmatrix} ,
	\end{equation}
	hence the functions $z_l''$ and $z_r''$ are well-defined $m$-almost everywhere by
	\begin{equation} \label{eq:determinants_2}
	z_l''(t) = \begin{pmatrix} - \sin \theta_l(t) \\ \cos \theta_l(t) \end{pmatrix} \theta_l'(t), \qquad \qquad
	z_r''(t) = \begin{pmatrix} - \sin \theta_r(t) \\ \cos \theta_r(t) \end{pmatrix} \theta_r'(t).
	\end{equation}
	In particular, they coincide $m$-almost everywhere. 
\end{proof}

\begin{proof} [Proof of Theorem~\ref{thm:kappa_zeta''}]
By Lemma~\ref{thm:increasing}, the Borel measure $\sigma$ on $J$ defined in \eqref{eq:sigma} is positive. Now, by the equalities in \eqref{eq:auxiliary_4}, for all $a,b \in J$, $a \leq b$ we have
\begin{align*}
& \sigma((a,b)) = \max\{ 0, \theta_l(b) - \theta_r(a) \}, \qquad \qquad & \sigma((a,b]) = \theta_r(b) - \theta_r(a), \\
& \sigma([a,b)) = \theta_l(b) - \theta_l(a), \qquad \qquad & \sigma([a,b]) = \theta_r(b) - \theta_l(a).
\end{align*}
The metric density associated with the absolutely continuous part of $\sigma$ with respect to the Lebesgue measure $m$ on $J$ is $\kappa$.

Next, we define the Borel measure $\sigma_r$ on $J$ as follows. For all $a,b \in J$, $a \leq b$ we define
\begin{align*}
& \sigma_r((a,b)) = \max\{ 0, \lim_{t \to b^{-}}\theta_r(t) - \theta_r(a) \}, \qquad \qquad & \sigma_r((a,b]) = \theta_r(b) - \theta_r(a), \\
& \sigma_r([a,b)) = \lim_{t \to b^{-}}\theta_r(t) - \theta_r(a), \qquad \qquad & \sigma_r([a,b]) = \theta_r(b) - \theta_r(a).
\end{align*}
The metric density associated with the absolutely continuous part of $\sigma_r$ with respect to the Lebesgue measure $m$ on $J$ coincides $m$-almost everywhere with $\theta_r'$.

For every $b \in J$ we consider the sequence of sets $\{ (b-\varepsilon, b] \colon \varepsilon > 0 \}$ that shrinks to $b$ nicely as in Definition~\ref{def:nicely}. On each of these sets, the Borel measure $\sigma - \sigma_r$ is zero. By Theorem~\ref{thm:derivative}, the metric density associated with the absolutely continuous part of $\sigma - \sigma_r$ with respect to the Lebesgue measure $m$ on $J$ is zero $m$-almost everywhere. Therefore, the functions $\kappa$ and $\theta'_r$ coincide $m$-almost everywhere. Analogously we prove that the functions $\kappa$ and $\theta'_l$ coincide $m$-almost everywhere. By Theorem~\ref{thm:derivative_z'} and the equalities in \eqref{eq:determinants_1} and \eqref{eq:determinants_2}, for $m$-almost every $t \in J$ we have 
\begin{equation*}
\theta_l'(t) = \det \begin{pmatrix} z_l'(t) & z_l''(t)  \end{pmatrix}, \qquad \qquad \theta_r'(t) = \det \begin{pmatrix} z_r'(t) & z_r''(t)  \end{pmatrix},
\end{equation*}
yielding the desired result.
\end{proof}

\bibliographystyle{amsplain}
\bibliography{mybibliography}

\end{document}